\documentclass{article}
\usepackage{fancyhdr}

\usepackage[british]{babel}
\usepackage[letterpaper,top=2cm,bottom=2cm,left=3cm,right=3cm,marginparwidth=1.75cm]{geometry}
\usepackage{amsmath, amsbsy, amssymb, amsthm, stmaryrd, bbm}
\usepackage{graphicx}
\usepackage{xspace}
\usepackage[T1]{fontenc}
\usepackage[utf8]{inputenc}
\usepackage{float}
\usepackage{caption}
\usepackage{quiver}
\usepackage{yhmath}
\usepackage[nottoc,numbib]{}
\usepackage{scalerel}
\usepackage{footmisc}
\usepackage[useregional]{datetime2}
\usepackage[colorlinks=true, allcolors=blue]{hyperref}
\usepackage[maxbibnames=99, style=alphabetic, backend=biber, url=false, doi=false, isbn=false, date=year, hyperref, backref, backrefstyle=none]{biblatex}
\AtEveryBibitem{\clearlist{language}}
\DefineBibliographyStrings{british}{
  backrefpage={$\uparrow$\!\!},
  backrefpages={$\uparrow$\!\!}
}
\usepackage{filecontents}
\DTMlangsetup[en-GB]{showdayofmonth=false}
\usepackage{tikz}

\usepackage{mathtools}

\newtheorem{proposition}{Proposition}[section]
\newtheorem{theorem}[proposition]{Theorem}
\newtheorem{lemma}[proposition]{Lemma}
\newtheorem{example}[proposition]{Example}
\newtheorem{definition}[proposition]{Definition}
\newtheorem{definition-lemma}[proposition]{Definition-Lemma}
\theoremstyle{definition}
\newtheorem{remark}[proposition]{Remark}

\newtheorem*{question*}{Questions}
\newtheorem*{lemma*}{Lemma}

\numberwithin{equation}{proposition}

\makeatletter
\newcommand{\etale}{\'etal\@ifstar{\'e}{e\xspace}}
\makeatother
\newcommand\blfootnote[1]{%
  \begingroup
  \renewcommand\thefootnote{}\footnote{#1}%
  \addtocounter{footnote}{-1}%
  \endgroup
}
\newcommand{\Addresses}{{
  \bigskip
\noindent\textsc{Department of Mathematics, Massachusetts Institute of Technology, USA,}\par\nopagebreak
\noindent Email: \texttt{tong.g.h.zhou@gmail.com}
  }}

\newcommand{\isoto}{\xrightarrow{\raisebox{-0.5ex}[0ex][0ex]{$\sim$}}}
\newcommand{\isoot}{\xleftarrow{\raisebox{-0.5ex}[0ex][0ex]{$\sim$}}}

\newcommand{\ZZ}{\mathbf{Z}}
\newcommand{\QQ}{\mathbf{Q}}

\newcommand{\FF}{\mathbf{F}}

\newcommand{\CF}{\mathcal{F}}

\newcommand{\Gm}{\mathbf{G}_m}
\newcommand{\X}{\times}
\newcommand{\OX}{\otimes}

\newcommand{\CCC}{\mathcal{C}}
\newcommand{\CH}{\mathcal{H}}
\newcommand{\CL}{\mathcal{L}}

\newcommand{\CK}{\mathcal{K}}
\newcommand{\AAA}{\mathbf{A}}
\newcommand{\DD}{\mathbb{D}}

\newcommand{\CA}{\mathcal{A}}
\newcommand{\CO}{\mathcal{O}}

\newcommand{\RHom}{\mathrm{RHom}}
\newcommand{\RHOM}{\mathrm{R\underline{Hom}}}
\newcommand{\Hom}{\mathrm{Hom}}

\newcommand{\tE}{\widetilde{E}}

\newcommand{\fib}{\mathrm{fib}}

\newcommand{\pr}{\mathrm{pr}}
\newcommand{\id}{\mathrm{id}}

\newcommand{\Mod}{\mathrm{Mod}}

\newcommand{\Spec}{\mathrm{Spec}}

\newcommand{\Vect}{\mathrm{Vect}_L}
\newcommand{\PrL}{\mathrm{Pr}^L_{\mathrm{st}}}
\newcommand{\CAlg}{\mathrm{CAlg}}
\newcommand{\simeqd}{\mathrel{\rotatebox[origin=c]{90}{$\backsimeq$}}}
\newcommand{\1}{\mathbbm{1}}
\newcommand{\Td}{\langle d\rangle}
\newcommand{\Tmd}{\langle -d\rangle}
\newcommand{\Tmdt}{\langle -\widetilde{d}\rangle}
\newcommand{\tF}{\widetilde{F}}
\newcommand{\Tdt}{\langle \widetilde{d}\rangle}
\newcommand{\CE}{\mathcal{E}}
\newcommand{\Ga}{\mathbf{G}_a}
\newcommand{\simd}{\mathrel{\rotatebox[origin=c]{90}{$\backsim$}}}
\newcommand{\tp}{\widetilde{p}}
\newcommand{\td}{\widetilde{d}}

\title{A motivic derived Fourier transform}
\author{Tong Zhou}
\date{}
\begin{document}
\maketitle
\begin{abstract}
\blfootnote{August 2026}
We construct a theory of Fourier transform for rational étale motives on derived vector bundles over derived Artin stacks over a field of positive characteristic, defined as the integral transform with kernel the Artin–Schreier motive. The basic properties are established: involutivity, base change, functoriality, duality, monodromicity, constructibility, and compatibilities with various natural transformations.
\end{abstract}
\renewcommand{\baselinestretch}{1.0}\normalsize
\setcounter{tocdepth}{2}
\tableofcontents
\renewcommand{\baselinestretch}{1.0}\normalsize
\vspace{10pt}
\section{Introduction}
Let $k$ be a field of characteristic $p>0$, and $L$ a finite extension of $\QQ$ containing a primitive $p$-th root of unity. We fix a primitive $p$-th root of unity $\zeta\in L$ throughout. For a derived Artin stack $X$ locally of finite type over $k$, denote by $D(X)$ the stable category of \'etale $L$-motives on $X$. In this work, we construct a theory of Fourier transform on derived vector bundles in this context.\\

The Fourier transform is defined as the integral transform with kernel the Artin–Schreier motive, which, by definition, is a non-trivial isotypic part of $f_*f^*\1_{\AAA^1}$ under the canonical $\mathrm{Gal}(f)(\simeq\ZZ/p)$-action, where $f: \AAA^1\rightarrow\AAA^1$ is the Artin–Schreier covering. Starting from this, we establish all the basic properties of the Fourier transform. The results are used in \cite{FYZ3}. In the case of vector bundles on schemes, a part of this theory has also been studied in \cite[\nopp §3.1, §3.2]{cass_vdhove_scholbach}.\\

The structure of this article is as follows: after preliminary discussion on \'etale motives and higher algebra in §\ref{sec_prelim}, we define and study the Artin–Schreier motives in detail in §\ref{sec_AS}. We study their behaviour under the tensor product and internal-Hom between each other, duality, scaling on $\AAA^1$, pushing forward to a point, and prove the additivity property. These are then used in §\ref{sec_fourier_on_vect}, where we construct the theory of Fourier transform on classical vector bundles over schemes. On the one hand, this case serves as the base for the case on derived vector bundles, on the other hand, most proofs here work \textit{verbatim} for the latter case, with the construction of involutivity as a notable exception. In §\ref{sec_F_dVect}, after a preliminary discussion on derived vector bundles, we construct the theory of Fourier transform in this context. The construction of involutivity is non-trivial and due to \cite[\nopp §A.3]{FYZ}. We add in some details and give a slightly different argument in the last gluing step (see Theorem \ref{thm_involutivity}), avoiding the use of $t$-structures on $D(X)$. Finally, we establish all basic properties of the Fourier transform in §\ref{subsec_F_properties_d}. We emphasise that, once the techniques involved in the construction of involutivity are properly understood, most of the properties follow easily. The exceptions are the compatibilities with proper base change and the Gysin map, which are non-trivial; nevertheless, the original proofs in \cite[\nopp §A.4, §6.4]{FYZ} apply in our context.\\

For the convenience of the reader, here we summarise the main results on the motivic derived Fourier transform. We refer to Conventions for any unspecified terminology. Let $E$ be a derived vector bundle of rank $d$ on a derived Artin stack $X$ locally of finite type over $k$, and $E'$ be its dual.

\begin{definition}
    The \underline{Fourier transform} $F$ is the functor $D(E)\rightarrow D(E'), A\mapsto p'_!(p^*A\OX m^*\CL)$, where $\CL$ is the Artin–Schreier motive, and $p'$, $p$, and $m$ are as in the following diagram:
\[\begin{tikzcd}
	& {E\times_XE'} && {\AAA^1} \\
	E && {E'}
	\arrow["{m\,\,\mathrm{(pairing)}}", from=1-2, to=1-4]
	\arrow["p"', from=1-2, to=2-1]
	\arrow["{p'}", from=1-2, to=2-3]
\end{tikzcd}\]
\end{definition}

\begin{theorem}
    (1) $\mathrm{(Involutivity)}$ There is a canonical pair $(\eta_E,\eta_{E'})$ of natural isomorphisms $\eta_E: F'\circ F\isoto [-1]^*\Tmd$ and $\eta_{E'}: F\circ F'\isoto [-1]^*\Tmd$, such that the following two diagrams commute: 
\[\begin{tikzcd}
	F & {F\circ F'\circ F[-1]^*\Td} & F & {F'} & {F'\circ F\circ F'[-1]^*\Td} & {F'}
	\arrow["{(-1)^d\cdot\id}", curve={height=24pt}, from=1-1, to=1-3]
	\arrow["{F\circ \eta_{E}}"', from=1-2, to=1-1]
	\arrow["\sim", from=1-2, to=1-1]
	\arrow["{\eta_{E'}\circ F}", from=1-2, to=1-3]
	\arrow["\sim"', from=1-2, to=1-3]
	\arrow["{(-1)^d\cdot\id}", curve={height=24pt}, from=1-4, to=1-6]
	\arrow["{F'\circ\eta_{E'}}"', from=1-5, to=1-4]
	\arrow["\sim", from=1-5, to=1-4]
	\arrow["\sim"', from=1-5, to=1-6]
	\arrow["{\eta_{E}\circ F'}", from=1-5, to=1-6]
\end{tikzcd}\]
    (2) $\mathrm{(Right\,\,adjoint)}$ There is a natural isomorphism $F\simeq p'_*(p^{!}(-)\OX m^*\CL)\Tmd$. If $p$ is smooth (which is equivalent to $E$ having tor-amplitude in $[0,+\infty)$), then there is a further natural isomorphism with $p'_*(p^*(-)\OX m^*\CL)$.\footnote{The isomorphism $F\simeq p'_*(p^*(-)\OX m^*\CL)$ is sometimes called the “miracle” of the Fourier transform (\cite[108]{katz_travaux_laumon}).}\\
    (3) $\mathrm{(Verdier\,\,duality)}$ There is a natural isomorphism $\DD_{E'/X}F\simeq (F\DD_{E/X})[-1]^*\Td$.\\
    (4) $\mathrm{(Convolution)}$ Let $(-)\ast(-): D(E)\X D(E)\rightarrow D(E)$ be the convolution with respect to the additive structure of $E$, similarly for $E'$. For $A,B\in D(E)$, we have natural isomorphisms:
    $$F(A\ast B)\simeq FA\OX FB\text{ and }F(A\OX B)\simeq (FA\ast FB)\Td.$$
    (5) $\mathrm{(Plancherel)}$ For $A,B\in D(E)$, we have a natural isomorphism
    $$\pi_{E'!}(FA\OX FB)\simeq\pi_{E!}(A\OX[-1]^*B)\Tmd$$
    where $\pi_E$ and $\pi_{E'}$ are the projections to $X$.\\
    (6) $\mathrm{(Monodromicity)}$ The Fourier transform restricts to an equivalence on monodromic objects (Definition \ref{def_mon_obj}): $D_{\mathrm{mon}}(E)\isoto D_{\mathrm{mon}}(E')$.\\
    (7) $\mathrm{(Constructibility)}$ The Fourier transform restricts to an equivalence on constructible objects: $D_{\mathrm{c}}(E)\isoto D_{\mathrm{c}}(E')$.\\
    (8) $\mathrm{(Compatibility \,\,with \,\,the \,\,homogeneous \,\,Fourier \,\,transform)}$ There is a commutative diagram
\[\begin{tikzcd}
	{D(E)} & {D(E')} \\
	{D(E/\Gm)} & {D(E'/\Gm)}
	\arrow["F", from=1-1, to=1-2]
	\arrow["{\rho_1^*}", from=2-1, to=1-1]
	\arrow["{F_1}"', from=2-1, to=2-2]
	\arrow["{\rho'^*_1}"', from=2-2, to=1-2]
\end{tikzcd}\]
where $\rho_1$ (resp. $\rho_1'$) is the quotient map $E\rightarrow E/\Gm$ (resp. $E'\rightarrow E'/\Gm$), and $F_1$ is the homogeneous Fourier transform (Definition \ref{def_F_1}).\\
    (9) $\mathrm{(Base\,\,change)}$ Let $f: \widetilde{X}\rightarrow X$ be a map locally of finite type of derived Artin stacks, let $\tE= E\X_X\widetilde{X}$ be the pullback. There are natural isomorphisms\footnote{Throughout this article, in a base change situation, we abuse notations and use $f$ to denote the maps from an object over $\widetilde{X}$ to the corresponding object over $X$ unless otherwise noted.}:
$$ (1)\,\, Ff_!\simeq f_!\tF,\,\,\,\, (2)\,\, f^*F\simeq \tF f^*,\,\,\,\, (3)\,\, f^!F\simeq \tF f^!,\,\,\,\, (4)\,\, Ff_*\simeq f_*\tF.$$
    (10) $\mathrm{(Functoriality)}$ Let $\alpha: \tE\rightarrow E$ be a map of derived vector bundles over $X$, and $\alpha'$ be the dual map. There are natural isomorphisms:
    $$ (1)\,\, F\alpha_!\simeq\alpha'^*\tF,\,\,\,\, (2)\,\, \alpha'_!F\Td\simeq \tF\alpha^*\Tdt, \,\,\,\, (3)\,\, F\alpha_*\Td\simeq\alpha'^!\tF\Tdt,\,\,\,\, (4)\,\, \alpha'_*F\simeq\tF\alpha^!.$$
    (11) $\mathrm{(Compatibility \,\,with \,\,proper \,\,base \,\,change)}$ Given a Cartesian square of derived vector bundles over $X$ and its dual:
\[\begin{tikzcd}[row sep=0.8cm, column sep=0.8cm]
	& D &&& {D'} & \\
	A && C & {A'} && {C'} \\
	& B &&& {B'}
	\arrow["c"', from=1-2, to=2-1]
	\arrow["d", from=1-2, to=2-3]
	\arrow["\square"{marking, allow upside down}, draw=none, from=2-1, to=2-3]
	\arrow["a"', from=2-1, to=3-2]
	\arrow["b", from=2-3, to=3-2]
	\arrow["{c'}", from=2-4, to=1-5]
	\arrow["{d'}"', from=2-6, to=1-5]
	\arrow["\square"{marking, allow upside down}, draw=none, from=2-6, to=2-4]
	\arrow["{a'}", from=3-5, to=2-4]
	\arrow["{b'}"', from=3-5, to=2-6]
\end{tikzcd}\]
Assume $a$ and $b$ admit presentations (see §\ref{subsec_dVect}), then we have a commutative diagram
\[\begin{tikzcd}
	{F_Cb^*a_!} & {F_Cd_!c^*} \\
	{b'_!a'^*F_A\langle d_B-d_C\rangle} & {d'^*c'_!F_A\langle d_A-d_D\rangle}
	\arrow["\simeq"{description}, draw=none, from=1-1, to=1-2]
	\arrow["\simeqd"{description}, draw=none, from=1-1, to=2-1]
	\arrow["\simeqd"{description}, draw=none, from=1-2, to=2-2]
	\arrow["\simeq"{description}, draw=none, from=2-1, to=2-2]
\end{tikzcd}\]
where the vertical isomorphisms are from Functoriality, and the horizontal ones are from proper base change.\\
    (12) $\mathrm{(Compatibility \,\,with \,\,the \,\,Gysin \,\,map)}$ Let $\alpha: \tE\rightarrow E$ be a quasi-smooth map of derived vector bundles over $X$. Assume $\alpha$ admits a presentation. Then we have a commutative diagram
\[\begin{tikzcd}
	{\alpha'_!F\Td} & {\alpha'_*F\Td} \\
	{\tF\alpha^*\Tdt} & {\tF\alpha^!\Td}
	\arrow["{\mathrm{can}(\alpha')}", from=1-1, to=1-2]
	\arrow["\simeqd"{description}, draw=none, from=1-1, to=2-1]
	\arrow["\simeqd"{description}, draw=none, from=1-2, to=2-2]
	\arrow["{[\alpha]}"', from=2-1, to=2-2]
\end{tikzcd}\]
where $\mathrm{can}(\alpha'): \alpha'_!\rightarrow\alpha'_*$ is the forget-support map,  $[\alpha]: \alpha^*\rightarrow\alpha^!\langle -\dim\alpha\rangle$ is the Gysin map, and the vertical isomorphisms are from Functoriality.
\end{theorem}

\section*{Conventions}
The conventions are as in \cites{FYZ, FK, FYZ3}. For derived Artin stacks, a reference is \cite[\nopp §2.2]{HAG}. Unless otherwise specified, all our derived Artin stacks are locally of finite type over $k$. The symbols $\AAA^1$ and $\Gm$ (resp. $(-)\X(-)$) without a subscript denote the objects (resp. fibre product) over $k$. By a point (resp. geometric point) we mean a map from the spectrum of a field (resp. separably closed field). We use the six-functor formalism of \'etale $\QQ$-motives; on schemes two references are \cites{ayoub2014, cisinskideglise2016} (see \cite[\nopp §3.1]{FK} for a comparison of different notions in the literature; for definiteness, we use the model of \cite{ayoub2014}), on derived Artin stacks a reference is \cite[Appendix]{khan_virtual}.  For a map $f$ of derived Artin stacks, we have the adjunction $f^*\dashv f_*$; for $f$ locally of finite type, we have the adjunction $f_!\dashv f^!$, the forget-support map $\mathrm{can}(f): f_!\rightarrow f_*$ (for $f$ representable by derived
Deligne–Mumford stacks), and the Gysin map $[f]: f^*\rightarrow f^!\langle -\dim f\rangle$ (for $f$ quasi-smooth), see \cite[Theorem A.7, §3.1]{khan_virtual}. For $i\in \ZZ$, we use the notation $\langle i\rangle=(i)[2i]$ throughout. For a derived Artin stack $X$, $X_{\mathrm{cl}}$ denotes its classical truncation. We have a canonical equivalence $D(X)\simeq D(X_{\mathrm{cl}})$. We denote the unit object in $D(X)$ by $\1_X$, and often omit the subscript when the context is clear.\\\\
We will need the fact that $D(-)$ forms a sheaf on the category of derived Artin stacks with respect to the smooth topology. This follows from \cite[Proposition 3.2.5]{khan_lisse}, whose hypothesis is satisfied because of \cite[Proposition 2.2.3.2]{HAG}.\\\\
Recall the notion of constructible motives (see \cite[\nopp §3.3]{FK}, where they are called geometric motives): (1) For a derived Artin stack $X$, the full subcategory $D_{c}(X)\subseteq D(X)$ of constructible objects consists of $A\in D(X)$ such that $f^*A$ is constructible for every smooth surjection $f:S\rightarrow X$ from a derived scheme, and it suffices to test for any such $f$. (2) For a derived scheme $S$, $D_{c}(S)$ is the thick subcategory of $D(S)$ generated by objects of the form $g_\sharp\1\langle i\rangle$, for $g$ ranging over all smooth maps from a derived scheme and $i$ ranging over $\ZZ$. Here $g_\sharp$ is the left adjoint of $g^*$. (3) The functors $\otimes$ and $\mathrm{R\underline{Hom}}$ preserve $D_c$. For a map $f$ of finite type of derived Artin stacks, $D_c$ is preserved under $f^*$ and $f^!$; if $f$ is representable (by derived schemes), then $f_*$ and $f_!$ also preserve $D_c$. (4) On a scheme which is finite dimensional, quasi-compact and quasi-separated, this notion of constructibility coincides with the usual one defined by stratification, and also with being compact (\cite[Proposition 8.3]{ayoub2014}, \cite{ruimytubach}). (5) The subcategory of dualisable objects is denoted by $D_{\mathrm{lis}}(X)$, it forms a thick subcategory of $D(X)$.\\\\
For a perfect complex $\CE$ (resp. derived vector bundle $E$) over a derived Artin stack $X$, its associated derived vector bundle (resp. perfect complex) is denoted by $E$ (resp. $\CE$), and $\CE$ corresponds to the sections of $E$. By the tor-amplitude of $E$ we mean the tor-amplitude of $\CE$. We denote by $\pi_E$ and $\iota_E$ the projection $E\rightarrow X$ and zero-section $X\rightarrow E$ respectively. The superscript “$'$” is used to denote objects associated to the dual bundle. Unless otherwise noted, $d$ ($\widetilde{d}$, $d'$, \textit{etc.}) denotes the rank of $E$ ($\tE$, $E'$, \textit{etc.}) (which is usually called the virtual rank, it is a locally constant function on the underlying topological space of $X$), and similarly for the Fourier transform $F$.\\\\
By a category we always mean an $\infty$-category. All diagrams displayed are commutative (meaning that there exists a witnessing homotopy), unless there appears a (co-)simplicial object.

\section*{Acknowledgements}
I heartily thank Tony Feng for suggesting this topic, teaching me the definition of the Artin–Schreier motive, and for subsequent discussions. I also thank Denis-Charles Cisinski, Samuel Muñoz-Echániz, Alexander Petrov, Jakob Scholbach and especially Zhiwei Yun for helpful discussions and feedback.

\section{Preliminaries}\label{sec_prelim}
\subsection{\'Etale motives}
We refer to Conventions for basic notions concerning \'etale motives. This subsection records a few lemmas to be used later.
\begin{lemma}\label{lem_point_conservative}
    Let $X$ be a scheme locally of finite dimension, or a derived Artin stack locally of finite type over a field. Then, pulling back to points forms a conservative family for $D(X)$. Here a point of $X$ means a map to $X$ from the spectrum of a field.
\end{lemma}
\begin{proof}
     The scheme case follows from \cite[Proposition 3.24]{ayoub2014}. The case of derived Artin stacks reduces to the case of schemes by considering a smooth surjection from a derived scheme $T$ and using $D(T)\simeq D(T_{\mathrm{cl}})$. 
\end{proof}

\begin{remark}\label{rmk_conservativity}
     For the subcategory of constructible objects $D_c(X)$, in fact pulling back to geometric points is already conservative, this is a direct consequence of \cite[Propositions 3.24, 3.20, Lemme 3.4]{ayoub2014}.
\end{remark}

\begin{lemma}[Künneth formula]\label{lem_kunneth}
    Let $f: X'\rightarrow X$ and $g: Y'\rightarrow Y$ be maps locally of finite type of schemes (resp. derived Artin stacks) over a base scheme (resp. derived Artin stack) $S$. Consider the following diagram, where fibre products are taken in schemes (resp. derived Artin stacks):
\[\begin{tikzcd}
	& {X'\X_S Y'} & \\
	{X'} & {X\X_S Y} & {Y'} \\
	X && Y
	\arrow[from=1-2, to=2-1]
	\arrow["{f\X_S g}", from=1-2, to=2-2]
	\arrow[from=1-2, to=2-3]
	\arrow["\square"{description, pos=0.58}, draw=none, from=2-1, to=2-2]
	\arrow["f"', from=2-1, to=3-1]
	\arrow[from=2-2, to=3-1]
	\arrow[from=2-2, to=3-3]
	\arrow["\square"{description, pos=0.58}, draw=none, from=2-3, to=2-2]
	\arrow["g", from=2-3, to=3-3]
\end{tikzcd}\]
Then, for $A\in D(X')$ and $B\in D(Y')$, there is a natural isomorphism $(f\X_S g)_!(A\boxtimes B)\simeq f_!A\boxtimes g_!B$.
\end{lemma}
\begin{proof}
    This is a formal consequence of proper base change and the projection formula (see, for example, \cite[Lemma 2.2.3]{JY21}.).
\end{proof}

\begin{lemma}\label{lem_hom_of_unit_conn_comp}
    Let $S$ be a connected scheme of finite dimension, or a connected derived Artin stack of finite type over a field. Then $\RHom_{D(S)}(\1_S,\1_S)$ is coconnective (\textit{i.e.} it lies in $\Vect^{\geq0}$), and $\Hom_{D(S)}(\1_S,\1_S)\simeq L$.
\end{lemma}
Recall that for a derived Artin stack $S$, being connected means the underlying topological space $|S|$ is connected.
\begin{proof}
    The scheme case is \cite[Proposition 11.1.a]{ayoub2014}. Let $S$ be a connected derived Artin stack of finite type over a field. We may assume $S$ equal to its classical truncation. Let $T\rightarrow S$ be a smooth surjection from a finite dimensional scheme. By construction, $D(S)=\lim_{\Delta}D(T_n)$, where $T_n:=T\X_ST\X_S\cdots\X_ST$ (($n+1$) times). So $\RHom(\1_S,\1_S)\simeq \lim_{\Delta}\RHom(\1_{T_n},\1_{T_n})$ in $\Vect$. By the scheme case, every $\RHom(\1_{T_n},\1_{T_n})$ is coconnective. This implies the coconnectivity of $\RHom(\1_S,\1_S)$. Further, $\Hom(\1_S,\1_S)\simeq H^0\lim_{\Delta}\RHom(\1_{T_n},\1_{T_n})\simeq\lim_{\Delta}H^0\RHom(\1_{T_n},\1_{T_n})$, where the latter limit is taken in $\Vect^{\heartsuit}$. By \cite[Proposition 11.1.a]{ayoub2014}, it is isomorphic to $\lim_{\Delta}L^{\pi_0(T_n)}$. As $S$ is connected, for every pair of connected components $T_{0,i}$ and $T_{0,j}$ of $T_0$, there exists a sequence $T_{0,i_0}=T_{0,i}, T_{0,i_1},\cdots, T_{0,i_m}=T_{0,j}$ such that $T_{0,i_a}\X_ST_{0,i_{a+1}}$ is non-empty for all $a\in\{0,...,m-1\}$. It follows that $\lim_{\Delta}L^{\pi_0(T_n)}\simeq L$. 
\end{proof}

\subsection{Higher algebra}\label{subsec_prelim_ha}
The reference for this paragraph is \cite[\nopp §4.5]{lurie_ha}. Also see \cite[\nopp §2.2, §4.1]{benzvifrancisnadler} and references therein. Denote by $\Vect$ the symmetric monoidal stable category of $L$-vector spaces. Then $\Vect$ naturally lies in $\CAlg(\PrL)$, the commutative algebra objects in $\PrL$. We can form $\Mod_{\Vect}(\PrL)$, the symmetric monoidal category of $\Vect$-modules in $\PrL$. An object in $\Mod_{\Vect}(\PrL)$ will be called an \textit{$L$-linear category}. Let $A\in\CAlg(\Vect)$. For every $\CCC\in\Mod_{\Vect}(\PrL)$ we have $\Mod_A(\CCC)$, the stable category of $A$-modules in $\CCC$. By \cite[Proposition 4.1.1]{benzvifrancisnadler}, there is a canonical equivalence $\Mod_A(\CCC)\simeq\Mod_A(\Vect)\otimes_{\Vect}\CCC$. If $A\rightarrow B$ is a map in $\CAlg(\Vect)$, there is an adjunction
\[\begin{tikzcd}
	{\Mod_A(\Vect)} && {\Mod_B(\Vect)}
	\arrow[""{name=0, anchor=center, inner sep=0}, "{(-)\OX_AB}", curve={height=-12pt}, from=1-1, to=1-3]
	\arrow[""{name=1, anchor=center, inner sep=0}, "{(-)_A}", curve={height=-12pt}, from=1-3, to=1-1]
	\arrow["\dashv"{anchor=center, rotate=-90}, draw=none, from=0, to=1]
\end{tikzcd}\]
where the lower arrow is forgetting along $A\rightarrow B$. This induces, via the above equivalence, an adjunction between $\Mod_A(\CCC)$ and $\Mod_B(\CCC)$.\\

This discussion applies in particular to a (usual) commutative algebra $A$ over $L$, by viewing $A$ as a commutative algebra object in $\Vect$.

\begin{lemma}\label{lem_decomp_mod}
    Let $A=A_1\X A_2$ be a product of (usual) commutative algebras over $L$, and $\CCC$ an $L$-linear category. Then there is a canonical equivalence $\Mod_A(\CCC)\simeq \Mod_{A_1}(\CCC)\X\Mod_{A_2}(\CCC)$, given by $M\mapsto (M\OX_AA_1,M\OX_AA_2)$, $(M_1)_A\oplus(M_2)_A\mapsfrom(M_1,M_2)$.
\end{lemma}
\begin{proof}
    Using $\Mod_A(\CCC)\simeq\Mod_A(\Vect)\otimes_{\Vect}\CCC$, it suffices to prove for $\CCC=\Vect$. Then  it is easy to check that the two functors are inverse to each other. We verify one direction: $M\mapsto (M\OX_AA_1,M\OX_AA_2)\mapsto (M\OX_AA_1)_A\oplus (M\OX_AA_2)_A\simeq M\OX_A(A_1\oplus A_2)\simeq M\OX_AA\simeq M$.
\end{proof}

We now discuss in detail one example of the above, to be used in §\ref{sec_AS}.\\

Fix an $L$-linear category $\CCC$. Let $A$ be the group algebra $L[\ZZ/p]\simeq \frac{L[t]}{t^p-1}$. We have $\frac{L[t]}{t^p-1}\simeq \frac{L[t]}{t-1}\X\frac{L[t]}{t-\zeta}\X\cdots\X\frac{L[t]}{t-\zeta^{p-1}}$. Denote the latter by $L_0\X L_1\X\cdots\X L_{p-1}$. Note each $L_i$ is isomorphic to $L$ as an $L$-vector space, via the map $L\rightarrow A\simeq L_0\X L_1\X\cdots\X L_{p-1}\xrightarrow{\mathrm{projection}}L_i$. Apply Lemma \ref{lem_decomp_mod}, we get, for every $M\in \Mod_{A}(\CCC)$, a canonical decomposition $M\simeq M_0\oplus M_1\oplus\cdots\oplus M_{p-1}$, where $M_i= (M\OX_A L_i)_{A}$. Each $M_i$ is canonically isomorphic to $M$ as objects in $\CCC$.\\


Now further assume $\CCC$ is a commutative algebra object in $L$-linear categories. Consider $A\otimes_L A':=\frac{L[t]}{t^p-1}\OX_L\frac{L[\tau]}{\tau^p-1}$. By \cite[Propositions 4.1.2, 4.1.1]{benzvifrancisnadler}, $\Mod_{A\otimes_L A'}(\CCC)\simeq\Mod_{A}(\CCC)\OX_\CCC\Mod_{A'}(\CCC)$. For $M\in\Mod_{A\otimes_L A'}(\CCC)$, let $N\otimes_\CCC N'$ be its image in $\Mod_{A}(\CCC)\OX_\CCC\Mod_{A'}(\CCC)$ under the above equivalence. We have a decomposition $N\otimes_\CCC N'\simeq (\oplus_iN_i)\OX_\CCC(\oplus_jN'_j)\simeq  \oplus_{(i,j)}N_i\OX_\CCC N'_j$, using the decomposition in the previous paragraph. Denote the image of $N_i\OX_\CCC N'_j$ in $\Mod_{A\otimes_L A'}(\CCC)$ by $M_{i,j}$, we get $M\simeq \oplus_{(i,j)}M_{i,j}$.\\

Assume $p>2$ for this paragraph. Consider the isomorphism $\frac{L[t]}{t^p-1}\OX_L\frac{L[\tau]}{\tau^p-1}\isoto \frac{L[u]}{u^p-1}\OX_L\frac{L[v]}{v^p-1}$, $t\otimes 1\mapsto u\otimes v$, $1\otimes\tau\mapsto u\otimes v^{-1}$. (Note the inverse is $t^{\frac{1}{2}}\otimes \tau^{\frac{1}{2}}\mapsfrom u\otimes 1$, $t^{\frac{1}{2}}\otimes \tau^{-\frac{1}{2}}\mapsfrom 1\otimes v$, where $\frac{1}{2}$ denotes the unique element $i$ in $\{0,1,...,p-1\}$ such that $2i=1$ in $\FF_p$.). This induces the commutative square
\[\begin{tikzcd}
	{\frac{L[t]}{t^p-1}\OX_L\frac{L[\tau]}{\tau^p-1}} & {\frac{L[u]}{u^p-1}\OX_L\frac{L[v]}{v^p-1}} \\
	{\prod_{(i,j)}\frac{L[t]}{t-\zeta^i}\OX_L\frac{L[\tau]}{\tau-\zeta^j}} & {\prod_{(a,b)}\frac{L[u]}{u-\zeta^a}\OX_L\frac{L[v]}{v-\zeta^b}}
	\arrow["\sim", from=1-1, to=1-2]
	\arrow["\simd"', from=1-1, to=2-1]
	\arrow["\simd", from=1-2, to=2-2]
	\arrow["\simeq"{description}, draw=none, from=2-1, to=2-2]
\end{tikzcd}\]
At the bottom, the $(i,j)$-th term is identified with the $(a,b)$-th term with $i=a+b$, $j=a-b$.\footnote{The indices are interpreted mod $p$.} This induces a similar identification between the decompositions of modules. Namely: given $M\in\Mod_{A\OX_LA'}(\CCC)$, we have $M\simeq\oplus_{(i,j)}M_{i,j}$ as before. Under the isomorphism of algebras $\frac{L[t]}{t-\zeta^i}\OX_L\frac{L[\tau]}{\tau-\zeta^j}\simeq\frac{L[u]}{u-\zeta^a}\OX_L\frac{L[v]}{v-\zeta^b}$, we can view $M$ as a module $M'$ over $\frac{L[u]}{u^p-1}\OX_L\frac{L[v]}{v^p-1}$, and have $M'\simeq\oplus_{(a,b)}M'_{a,b}$. Then, for $i=a+b$, $j=a-b$, $M_{i,j}$ and $M'_{a,b}$ correspond to each other under the isomorphism of algebras.

\section{The Artin–Schreier motives}\label{sec_AS}
Let $f:\AAA\!'^1=\Spec(k[t])\rightarrow\AAA^1=\Spec(k[x])$, $t^p-t\mapsfrom x$ be the Artin–Schreier covering. This is an \'etale Galois covering with Galois group $\ZZ/p$. The $\ZZ/p$-action on $\AAA\!'^1$ induces an action on $f^*\1$, hence on $f_*f^*\1$. Thus $f_*f^*\1$ is naturally an object of $\Mod_{L[\ZZ/p]}(D(\AAA^1))$. By the discussion in §\ref{subsec_prelim_ha}, we have a decomposition $f_*f^*\1\simeq \CL_0\oplus \CL_1\oplus\cdots\oplus\CL_{p-1}$ in $\Mod_{L[\ZZ/p]}(D(\AAA^1))$, where $1\in\ZZ/p$ acts by multiplication by $\zeta^i$ on $\CL_i$. We will abuse notations and denote the underlying object in $D(\AAA^1)$ of $\CL_i$ by the same letter. Note $f_*f^*\1\in D(\AAA^1)$ lies in $D_{\mathrm{lis}}(\AAA^1)$, the thick subcategory of dualisable objects (by \cite[Proposition 6.3.21]{cisinskideglise2016}), so each $\CL_i$ is dualisable.

\begin{definition}\label{def_as_motives}
    For $1\leq i\leq p-1$, the \underline{$i$-th Artin–Schreier motive} is $\CL_i\in D_{\mathrm{lis}}(\AAA^1)$ in the above decomposition. We will call $\CL_1$ “the” Artin–Schreier motive and denote it also by $\CL$.
\end{definition}

This section is devoted to establishing basic properties of Artin–Schreier motives. We first discuss a toy situation, which will also serve as the base case. We remind the reader that the subscript index in $\CL_i$ is interpreted mod $p$.

\subsection{Over a geometric point}\label{subsec_over_a_point}
Let $i: x\rightarrow \AAA^1$ be a geometric point. In this subsection, we analyse $i^*f_*f^*\1$. For the rest of this subsection, we denote $\AAA\!'^1\X_{\AAA^1}x\rightarrow x$ also by $f$.\\

Consider sheaves of $L$-vector spaces, denote the unit object by $L$. We have a decomposition $f_*f^*L\simeq L_0\oplus L_1\oplus\cdots\oplus L_{p-1}$ in $\Mod_{L[\ZZ/p]}(\Vect)$. Under this decomposition, the unit map $L\rightarrow f_*f^*L$ of the adjunction is an isomorphism to $L_0$, and the multiplication map $f_*f^*L\OX f_*f^*L\rightarrow f_*(f^*L\OX f^*L)\simeq f_*f^*L$ maps $L_i\OX L_j$ isomorphically to $L_{i+j}$.\\

Now let $\CCC$ be an $L$-linear category and consider sheaves valued in $\CCC$. Denote the unit object by $\1$. The sheaf operations are all induced from sheaves of $L$-vector spaces, so, by the equivalence $\Mod_{L[\ZZ/p]}(\CCC)\simeq\Mod_{L[\ZZ/p]}(\Vect)\otimes_{\Vect}\CCC$ (§\ref{subsec_prelim_ha}), the $\Vect$ case induces $f_*f^*\1\simeq L'_0\oplus L'_1\oplus\cdots\oplus L'_{p-1}$, with $L'_i$ being $L_i\OX \1$ under the above equivalence. Similarly, the unit map $\1\rightarrow f_*f^*\1$ is an isomorphism to $L'_0$, and the multiplication map $f_*f^*\1\OX f_*f^*\1\rightarrow f_*f^*\1$ maps $L'_i\OX L'_j$ isomorphically to $L'_{i+j}$.


\subsection{Properties of Artin–Schreier motives}
We now return to the set-up at the beginning of this section, and analyse $f_*f^*\1$. The unit map $\1\rightarrow f_*f^*\1$ and the multiplication map $f_*f^*\1\OX f_*f^*\1\rightarrow f_*f^*\1$ respect the $L[\ZZ/p]$-module structures ($\ZZ/p$ acts by the identity on $\1$ and $g\cdot (a\OX b)=ga\OX gb$ on $f_*f^*\1\OX f_*f^*\1$). So, by Lemma \ref{lem_decomp_mod}, $\1\rightarrow f_*f^*\1$ factors through $\CL_0$, and $f_*f^*\1\OX f_*f^*\1\rightarrow f_*f^*\1$ maps the factor $\CL_i\OX\CL_j$ to $\CL_{i+j}$.


\begin{lemma}\label{lem_L_products}
    With notations as above, the maps $\1\rightarrow\CL_0$ and  $\CL_i\OX\CL_j\rightarrow\CL_{i+j}$ are isomorphisms.
\end{lemma}
\begin{proof}
    By Remark \ref{rmk_conservativity}, it suffices to show these are isomorphisms at every geometric point $i: x\rightarrow \AAA^1$. We have a Cartesian diagram:
\[\begin{tikzcd}
	{\AAA\!'^1\X_{\AAA^1}x} & {\AAA\!'^1} \\
	x & {\AAA^{1}}
	\arrow[from=1-1, to=1-2]
	\arrow[""', from=1-1, to=2-1]
	\arrow["\square"{description, pos=0.45}, draw=none, from=1-1, to=2-2]
	\arrow["f", from=1-2, to=2-2]
	\arrow["i"', from=2-1, to=2-2]
\end{tikzcd}\]
As this is a pullback of $\ZZ/p$-torsors, $i^*$ commutes with decompositions of $L[\ZZ/p]$-modules. By the discussion in §\ref{subsec_over_a_point} (taking $\CCC$ to be $D(x)$), the conclusion follows.
\end{proof}

\begin{remark}\label{rmk_stalk_1}
    The proof also shows that, for every geometric point $i: x\rightarrow \AAA^1$ and $\CL_i$, we have $i^*\CL_i\simeq\1_x$. In fact, the same argument works for $\Spec(K)\rightarrow\AAA^1$ whenever $K$ is a splitting field for $t^p-t-\varphi(x)$, where $\varphi(x)$ denotes the image of $x\in k[x]$ in $K$.
\end{remark}

\begin{lemma}\label{lem_L_dual}
    For every $0\leq i\leq p-1$, we have $\CL_i\OX\CL_{p-i}\simeq \1$. Consequently, $\CL_{p-i}$ is canonically isomorphic to the dual $\CL_i^\vee$($:=\RHOM(\CL_i,\1_{\AAA^1})$).
\end{lemma}
\begin{proof}
    The first claim follows from Lemma \ref{lem_L_products}, the second follows from general facts about duality in a closed symmetric monoidal category (see, for example, \cite[\nopp 02DY]{kerodon}). 
\end{proof}

\begin{lemma}\label{lem__1_L}
    For $1\leq a\leq p-1$, denote by $[a]$ the map $\AAA^1\rightarrow\AAA^1$, $k[x]\mapsfrom k[x]$, $ax\mapsfrom x$. Then $[a]^*\CL_i\simeq\CL_{ai}$ for every $0\leq i\leq p-1$.
\end{lemma}
\begin{proof}
    Consider the following diagram:
\[\begin{tikzcd}
	{\AAA\!'^1} & {\AAA\!'^1} \\
	{\AAA^{1}} & {\AAA^{1}}
	\arrow["{[a]}", from=1-1, to=1-2]
	\arrow["f"', from=1-1, to=2-1]
	\arrow["\square"{description, pos=0.55}, draw=none, from=1-1, to=2-2]
	\arrow["f", from=1-2, to=2-2]
	\arrow["{[a]}"', from=2-1, to=2-2]
\end{tikzcd}\]
Note that it is Cartesian, and intertwines the $\ZZ/p$-actions under the map $\ZZ/p\rightarrow\ZZ/p: i\mapsto ai$, \textit{i.e.}, for a point $x$ in the left top $\AAA\!'^1$, we have $[a]i\cdot x=ai\cdot [a]x$. It follows that in the decomposition of the $f_*f^*\1$’s, the $\CL_i$ factor on the right corresponds to the $\CL_{ai}$ factor on the left.
\end{proof}

\begin{lemma}\label{lem_coho_L_zero}
    For every $i\neq0$, we have $p_*\CL_i=0$ and $p_!\CL_i=0$, where $p: \AAA^1\rightarrow \Spec(k)$ is the structure map.
\end{lemma}
\begin{proof}
    We show for $p_*$, the case $p_!$ is similar. Apply $p_*$ to $\1\rightarrow f_*f^*\1$, we get $p_*\1\rightarrow p_*f_*f^*\1\simeq p_*\CL_0\oplus p_*\CL_1\oplus\cdots \oplus p_*\CL_{p-1}$, which is an isomorphism to $p_*\CL_0$. We have $p_*\1\simeq p_*\CL_0\simeq \1$ and $p_*f_*f^*\1\simeq p_*f_*\1_{\AAA\!'^1}\simeq \1$. As $\1$ is indecomposable in $D(\Spec(k))$,\footnote{In fact, $\1_X$ is indecomposable for every connected scheme $X$ of finite dimension, by Lemma \ref{lem_hom_of_unit_conn_comp}.} the conclusion follows.
\end{proof}

\begin{lemma}
    For every $0\leq i,j\leq p-1$, $\RHom(\CL_i,\CL_j)\simeq L$ if $i=j$, and $0$ otherwise.
\end{lemma}
\begin{proof}
    We have $\RHom(\CL_i,\CL_j)\simeq\RHom(\1,\RHOM(\CL_i,\CL_j))\simeq \RHom(\1,\CL_{p-i}\OX\CL_j)\simeq \RHom_{\Spec(k)}(\1,p_*\CL_{p-i+j})$, where we used Lemma \ref{lem_L_dual} in the second step, and adjunction and Lemma \ref{lem_L_products} in the third. The claim now follows from Lemma \ref{lem_L_products}, Lemma \ref{lem_coho_L_zero} and \cite[Theorem 5.2.2]{cisinskideglise2016}.
\end{proof}

\begin{proposition}[Additivity]\label{prop_L_additive}
    For every $0\leq i\leq p-1$, $\CL_i$ is additive: let $q_1$, $q_2$ be the two projections $\AAA^1\X\AAA^1\rightarrow\AAA^1$ and $\mathrm{add}: \AAA^1\X\AAA^1\rightarrow\AAA^1$ be the addition map, then there is a canonical isomorphism $q_1^*\CL_i\OX q_2^*\CL_i\simeq \mathrm{add}^*\CL_i$ in $D(\AAA^1\X\AAA^1)$.
\end{proposition}

This is proved in \cite[Lemma 3.3]{cass_vdhove_scholbach} and \cite[Lemma 5.1.1]{FYZ3}. We record another proof in the case $p>2$ as a possibly useful additional reference.

\begin{proof}[Proof (for $p>2$)]
By Lemma \ref{lem_L_products}, it suffices to prove for $i=1$. Consider the following Cartesian square (here $\AAA^1_x=\Spec(k[x])$, \textit{etc.}):
\[\begin{tikzcd}[row sep=0.9cm, column sep=0.5cm]
	& W & \\
	{\CA_t} && {\CA_\tau} \\
	& {\AAA_x^{1}\X\AAA_y^{1}}
	\arrow[from=1-2, to=2-1]
	\arrow[from=1-2, to=2-3]
	\arrow["\pi", from=1-2, to=3-2]
	\arrow[from=2-1, to=3-2]
	\arrow[from=2-3, to=3-2]
\end{tikzcd}\]
where $\CA_t\rightarrow\AAA_x^{1}\X\AAA_y^{1}$ is the pullback under the projection $\AAA_x^{1}\X\AAA_y^{1}\rightarrow\AAA^1_x$ of the Artin–Schreier covering $\AAA^1_t\rightarrow\AAA^1_x$, and similarly for $\CA_\tau$ (in the $y$-coordinate). Then $\pi$ is a $\ZZ/p\X\ZZ/p$-torsor, and we have a decomposition $\pi_*\1\simeq \oplus_{(i,j)}(\pi_*\1)_{(i,j)}$ as $\frac{L[t]}{t^p-1}\OX_L\frac{L[\tau]}{\tau^p-1}$-modules in $D(\AAA_x^{1}\X\AAA_y^{1})$ (recall the discussion in §\ref{subsec_prelim_ha}). Here, we abuse notations and use the letters $t$ and $\tau$ also for the variables in the group algebras. We have $(\pi_*\1)_{(i,j)}\simeq \CL_i\boxtimes\CL_j\in \Mod_{\frac{L[t]}{t^p-1}\OX_L\frac{L[\tau]}{\tau^p-1}}(D(\AAA_x^{1}\X\AAA_y^{1}))$.\\\\
Consider another copy of the above diagram, with coordinates denoted by new letters $z$, $w$, $u$, $v$:
\[\begin{tikzcd}[row sep=0.9cm, column sep=0.5cm]
	& {W'} & \\
	{\CA_u} && {\CA_v} \\
	& {\AAA_z^{1}\X\AAA_w^{1}}
	\arrow[from=1-2, to=2-1]
	\arrow[from=1-2, to=2-3]
	\arrow["{\pi'}", from=1-2, to=3-2]
	\arrow[from=2-1, to=3-2]
	\arrow[from=2-3, to=3-2]
\end{tikzcd}\]
Let $c$ be the isomorphism $\AAA_x^{1}\X\AAA_y^{1}\rightarrow\AAA_z^{1}\X\AAA_w^{1}$, $(x,y)\mapsto (x+y,x-y)$, and $\mathrm{add}: \AAA_x^{1}\X\AAA_y^{1}\rightarrow\AAA^1$ the addition. Then $\mathrm{add}^*\CL_1\simeq c^*\pr_z^*\CL_1$, where $\pr_z$ denotes the projection to the $\AAA^1_z$-factor. Note $\pr_z^*\CL_1$ is the $(1,0)$-part of the decomposition $\pi'_*\1\simeq\oplus_{(a,b)}(\pi'_*\1)_{(a,b)}$ as $\frac{L[u]}{u^p-1}\OX_L\frac{L[v]}{v^p-1}$-modules in $D(\AAA_z^{1}\X\AAA_w^{1})$.\\\\
We have the following commutative diagram:
\[\begin{tikzcd}
	W & {W'} \\
	\\
	{\AAA_x^{1}\X\AAA_y^{1}} & {\AAA_z^{1}\X\AAA_w^{1}}
	\arrow["\sim"', from=1-1, to=1-2]
	\arrow["{c'}", from=1-1, to=1-2]
	\arrow["\pi"', from=1-1, to=3-1]
	\arrow["{\pi'}", from=1-2, to=3-2]
	\arrow["c"', from=3-1, to=3-2]
	\arrow["\sim", from=3-1, to=3-2]
\end{tikzcd}\]
where $c': (x,y,t,\tau)\mapsto (z,w,u,v)=(x+y,x-y,t+\tau,t-\tau)$. One verifies that this diagram is compatible with the $\ZZ/p\X\ZZ/p$-torsor structures under the isomorphism $\ZZ/p\X\ZZ/p\isoto\ZZ/p\X\ZZ/p, (n,m)\mapsto(n+m,n-m)$, which, in terms of algebras, is $\frac{L[t]}{t^p-1}\OX_L\frac{L[\tau]}{\tau^p-1}\isoto\frac{L[u]}{u^p-1}\OX_L\frac{L[v]}{v^p-1}, t\otimes1\mapsto u\otimes v, 1\otimes \tau\mapsto u\otimes v^{-1}$. By the last paragraph in §\ref{subsec_prelim_ha}, $(\pi_*\1)_{(i,j)}\simeq c^*(\pi'_*\1)_{(a,b)}$ for $i=a+b$, $j=a-b$. Combined with the above, we get $\CL_1\boxtimes\CL_1\simeq(\pi_*\1)_{(1,1)}\simeq c^*(\pi'_*\1)_{(1,0)}\simeq c^*\pr_z^*\CL_1\simeq \mathrm{add}^*\CL_1$.
\end{proof}

\section{The motivic Fourier transform on vector bundles}\label{sec_fourier_on_vect}
In this section, we construct the theory of Fourier transform for \'etale $L$-motives on vector bundles over finite dimensional schemes over $k$. Here a vector bundle means the total space of a locally free sheaf locally of finite rank. This will serve as the base for the theory on derived vector bundles. See Conventions for the notations.

\begin{definition}\label{def_fourier_transform}
    Let $E\rightarrow X$ be a vector bundle on a finite dimensional scheme over $k$, and $E'$ be the dual bundle. The \underline{(motivic) Fourier transform} $F$ is the functor $D(E)\rightarrow D(E')$ defined as the integral transform with kernel $\CL$ (Definition \ref{def_as_motives}), \textit{i.e.}, $F:=p'_!(p^*(-)\OX m^*\CL)$, where $p'$, $p$, and $m$ are as in the following diagram:
\[\begin{tikzcd}
	& {E\times_XE'} && {\AAA^1} \\
	E && {E'}
	\arrow["{m\,\,\mathrm{(pairing)}}", from=1-2, to=1-4]
	\arrow["p"', from=1-2, to=2-1]
	\arrow["{p'}", from=1-2, to=2-3]
\end{tikzcd}\]
\end{definition}

\begin{proposition}[Involutivity]\label{prop_involutivity}
    There is a natural isomorphism $F'F\simeq [-1]^*\Tmd$, where $[-1]: E\rightarrow E$ is the map of multiplication by $-1$.
\end{proposition}
\begin{proof}
    The composition $F'F: D(E)\rightarrow D(E)$ is the integral transform with kernel $\pr_{13,!}(\pr_{12}^*m^*\CL\OX\pr_{23}^*m^*\CL)$, where the maps are as in the following diagram:
\[\begin{tikzcd}
	&& {E\X_X E'\X_X E} && \\
	& {E\X_X E'} & {E\X_X E} & {E'\X_X E} \\
	E && {E'} && E
	\arrow["{\pr_{12}}"', from=1-3, to=2-2]
	\arrow["{\pr_{13}}", from=1-3, to=2-3]
	\arrow["{\pr_{23}}", from=1-3, to=2-4]
	\arrow[from=2-2, to=3-1]
	\arrow[from=2-2, to=3-3]
	\arrow[curve={height=-6pt}, from=2-3, to=3-1]
	\arrow[curve={height=6pt}, from=2-3, to=3-5]
	\arrow[from=2-4, to=3-3]
	\arrow[from=2-4, to=3-5]
\end{tikzcd}\]
It suffices to provide an isomorphism between this sheaf and $\delta^a_*\1_{\Delta^a}\Tmd$, where $\delta^a: \Delta^a(\simeq E)\hookrightarrow E\X_XE: v\mapsto (v,-v)$ is the anti-diagonal map. Consider the following two diagrams:
\[\begin{tikzcd}
	& {E\X_X E'\X_X E} &&&& \\
	& {\AAA^1\X\AAA^1} && {E\X_X E'\X_X E} & {E\X_XE'} & {\AAA^1} \\
	{\AAA^1} & {\AAA^1} & {\AAA^1} & {E\X_X E} & E
	\arrow["\alpha", from=1-2, to=2-2]
	\arrow["{m_{12}}"', from=1-2, to=3-1]
	\arrow["{m_{23}}", from=1-2, to=3-3]
	\arrow["{\pr_1}", from=2-2, to=3-1]
	\arrow["{\mathrm{add}}", from=2-2, to=3-2]
	\arrow["{\pr_2}"', from=2-2, to=3-3]
	\arrow["{\mathrm{add}\X \id}", from=2-4, to=2-5]
	\arrow[""{name=0, anchor=center, inner sep=0}, "{\pr_{13}}", from=2-4, to=3-4]
	\arrow["m", from=2-5, to=2-6]
	\arrow[""{name=1, anchor=center, inner sep=0}, "\pr", from=2-5, to=3-5]
	\arrow["{\mathrm{add}}", from=3-4, to=3-5]
	\arrow["\square"{description, pos=0.55}, draw=none, from=0, to=1]
\end{tikzcd}\]
Here $\mathrm{add}\X \id$ denotes $(v_1,w,v_2)\mapsto (v_1+v_2,w)$, and $\alpha$ is the map induced by the two pairings $m_{12}$ and $m_{23}$. Note $\mathrm{add}\circ\alpha=m\circ(\mathrm{add}\X \id)$. By the additivity of the Artin–Schreier motive (Proposition \ref{prop_L_additive}) and proper base change, $\pr_{13,!}(\pr_{12}^*m^*\CL\OX\pr_{23}^*m^*\CL)\simeq \mathrm{add}^*\pr_!m^*\CL$. So it suffices to provide an isomorphism $\iota_{E*}\1_{X}\Tmd\simeq\pr_!m^*\CL$, where $\iota_{E}: 0_E(=X)\hookrightarrow E$ is the zero-section of $E$.\\\\
We first provide a map. Let $i$ be the closed immersion $0_E\X_XE'\hookrightarrow E\X_XE'$. Remark \ref{rmk_stalk_1} gives a canonical isomorphism $i_0^*\CL\simeq \1_0$, where $i_0$ is the inclusion of $0:=\Spec(k[x]/x)$ in $\AAA^1$.  Via $i^*m^*$, this induces $i^*m^*\CL\simeq\1_{0_E\X_XE'}$. By adjunction, we get $m^*\CL\rightarrow i_{*}\1_{0_E\X_XE'}$. Apply $\pr_!$, we get $\pr_!m^*\CL\rightarrow \pr_! i_*\1_{0_E\X_XE'}\simeq  \iota_{E*}\pr_!\1_{0_E\X_XE'}\simeq \iota_{E*}\1_{X}\Tmd$.\\\\
To see this is an isomorphism, it suffices to check at every point $i_v: v\rightarrow E$ (Lemma \ref{lem_point_conservative}). If $i_v(v)\notin0_E$, then $v\X_X E'$ can be written as a product $\AAA^1_{k(v)}\X_{k(v)} \AAA_{k(v)}^{d'-1}$ with $m|_{v\X_X E'}$ factoring through the projection to $\AAA^1_{k(v)}$, so $(\pr_!m^*\CL)|_v=0$ by the Künneth formula and Lemma \ref{lem_coho_L_zero}; if $i_v(v)\in0_E$, then $(\pr_!m^*\CL)_v\simeq \pr_{v!}((m^*\CL)|_{v\X_X E'})\simeq \pr_{v!}\1_{v\X_X E'}\simeq \1_v\Tmd$. This completes the proof.
\end{proof}

\begin{remark}\label{rmk_sign_classical}
    The same construction applied to $E'$ gives $FF'\simeq [-1]^*\Tmd$. Denote $F'F\isoto [-1]^*\Tmd$ by $\eta_E$ and $FF'\isoto [-1]^*\Tmd$ by $\eta_{E'}$. Then the following diagrams commute: 
\[\begin{tikzcd}
	F & {F\circ F'\circ F[-1]^*\Td} & F & {F'} & {F'\circ F\circ F'[-1]^*\Td} & {F'}
	\arrow["{(-1)^d}", curve={height=24pt}, from=1-1, to=1-3]
	\arrow["{F\circ \eta_{E}}"', from=1-2, to=1-1]
	\arrow["\sim", from=1-2, to=1-1]
	\arrow["{\eta_{E'}\circ F}", from=1-2, to=1-3]
	\arrow["\sim"', from=1-2, to=1-3]
	\arrow["{(-1)^d}", curve={height=24pt}, from=1-4, to=1-6]
	\arrow["{F'\circ\eta_{E'}}"', from=1-5, to=1-4]
	\arrow["\sim", from=1-5, to=1-4]
	\arrow["\sim"', from=1-5, to=1-6]
	\arrow["{\eta_{E}\circ F'}", from=1-5, to=1-6]
\end{tikzcd}\]
where $(-1)^d$ denotes multiplication by $(-1)^d$.
\end{remark}
\begin{proof}
We sketch a proof of the commutativity of the left diagram (the right one follows, by switching $E$ and $E'$). First one shows it must commute up to a scalar (this is proved in a more general context in Remark \ref{rmk_inv_datum_basics}.3), so it suffices to do the computation with respect to the object $\1_E$. As the Fourier transform preserves box products (by the Künneth formula and the additivity of the Artin–Schreier motive) and commutes with base change (Lemma \ref{lem_F_base_change}.2), the question reduces to a computation for $E=\AAA^1_{X}$, for $X$ a geometric point. We want to show the diagram commutes in this case (so $d=1$). Use coordinates $x$, $u$, $y$, and $v$ on $E$, $E'$, $E''(\simeq E)$ and $E'''(\simeq E')$, and denote them by $\AAA^1_x$, $\AAA^1_u$, $\AAA^1_y$, and $\AAA^1_v$, respectively. Denote by $\CL_{xu}$ the pullback of $\CL$ to $\AAA^2_{xu}:=\AAA^1_x\X_X\AAA^1_u$ under the pairing, and similarly for $\CL_{uy}$ and $\CL_{yv}$. Then $F\circ\eta_E$ corresponds to an isomorphism $p_{xv!}(p_{xu}^*\CL_{xu}\otimes p_{uy}^*\CL_{uy}\otimes p_{yv}^*\CL_{yv})\simeq q_{u!}p_{xuv!}(p_{xu}^*\CL_{xu}\otimes p_{uy}^*\CL_{uy}\otimes p_{yv}^*\CL_{yv})\simeq q_{u!}(\CL_{xu}\otimes \Delta^a_{uv}\langle -1\rangle)\simeq \CL_{-xv}\langle -1\rangle$, where the maps are as in the following diagram, $\Delta^a_{uv}$ is the $*$-extension of the unit on the anti-diagonal, and $\CL_{-xv}$ is the pullback under the negative pairing. Similarly for $\eta_{E'}\circ F$. We want to show that the two isomorphisms $p_{xv!}(p_{xu}^*\CL_{xu}\otimes p_{uy}^*\CL_{uy}\otimes p_{yv}^*\CL_{yv})\simeq\CL_{-xv}\langle -1\rangle$ differ by $-1$. 
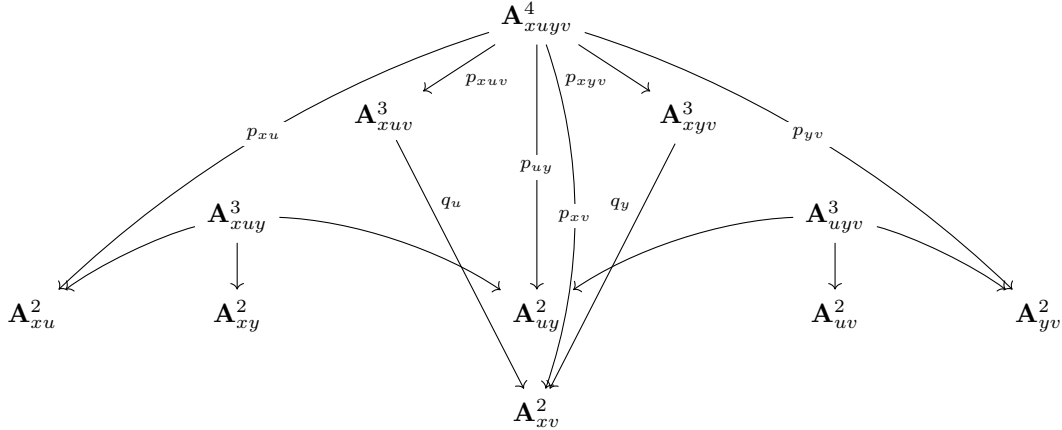
\[\begin{tikzcd}
	&&&& {\AAA^4_{xuyv}} &&&& \\
	&&& {\AAA^3_{xuv}} && {\AAA^3_{xyv}} \\
	&& {\AAA^3_{xuy}} &&&& {\AAA^3_{uyv}} \\
	{\AAA^2_{xu}} && {\AAA^2_{xy}} && {\AAA^2_{uy}} && {\AAA^2_{uv}} && {\AAA^2_{yv}} \\
	&&&& {\AAA^2_{xv}}
	\arrow["{p_{xuv}}", from=1-5, to=2-4]
	\arrow["{p_{xyv}}"', from=1-5, to=2-6]
	\arrow["{p_{xu}}"{description}, curve={height=18pt}, from=1-5, to=4-1]
	\arrow["{p_{uy}}"{description}, from=1-5, to=4-5]
	\arrow["{p_{yv}}"{description}, curve={height=-18pt}, from=1-5, to=4-9]
	\arrow["{p_{xv}}"{description}, curve={height=-18pt}, from=1-5, to=5-5]
	\arrow["{q_{u}}"{pos=0.3}, from=2-4, to=5-5]
	\arrow["{q_{y}}"'{pos=0.3}, from=2-6, to=5-5]
	\arrow[curve={height=6pt}, from=3-3, to=4-1]
	\arrow[from=3-3, to=4-3]
	\arrow[curve={height=-12pt}, from=3-3, to=4-5]
	\arrow[curve={height=12pt}, from=3-7, to=4-5]
	\arrow[from=3-7, to=4-7]
	\arrow[curve={height=-6pt}, from=3-7, to=4-9]
\end{tikzcd}\]
The scalar can be computed at the origin $x=v=0$ in $\AAA^2_{xv}$. One then checks that the question is reduced to the following: consider the following diagram:
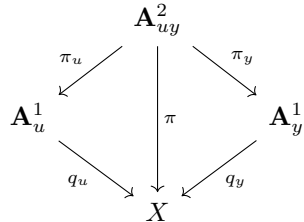
\[\begin{tikzcd}
	& {\AAA_{uy}^{2}} & \\
	{\AAA_u^{1}} && {\AAA_y^{1}} \\
	& X
	\arrow["{\pi_u}"', from=1-2, to=2-1]
	\arrow["{\pi_y}", from=1-2, to=2-3]
    \arrow["\pi", from=1-2, to=3-2]
	\arrow["q_u"', from=2-1, to=3-2]
	\arrow["q_y", from=2-3, to=3-2]
\end{tikzcd}\]
Corresponding to $\eta_E$ and $\eta_{E'}$ are the isomorphisms $\pi_!\CL_{uy}\simeq q_{u!}\pi_{u!}\CL_{uy}\simeq \1_X\langle -1\rangle$ and $\pi_!\CL_{uy}\simeq q_{y!}\pi_{y!}\CL_{uy}\simeq \1_X\langle -1\rangle$. We need to show they differ by $-1$.\\\\
Denote by $\sigma$ the automorphism $(u,y)\mapsto (y,u)$ of $\AAA^2_{uy}$. This is an automorphism over the pairing $m: \AAA^2_{uy}\rightarrow\AAA^1_X$, so equips $m_!m^*\CL$ with a $\sigma$-action. The assertion is equivalent to that the induced action on $q_!m_!m^*\CL\simeq\1_X\langle -1\rangle$ is $-1$, where $q$ is the map $\AAA^1_X\rightarrow X$. The projection formula gives $m_!m^*\CL\simeq (m_!m^*\1_{\AAA^1_X})\otimes \CL$. Let's first study $m_!m^*\1_{\AAA^1_X}$. Let $j: \AAA^1_X\backslash0\rightarrow\AAA^1_X$ and $i: 0\hookrightarrow \AAA^1_X$ be the immersions. Over $\AAA^1_X\backslash0$, $m$ is a trivial $\Gm$-torsor, so $j^*m_!m^*\1_{\AAA^1_X}\simeq \1_{\AAA^1_X\backslash0}[-1]\oplus \1_{\AAA^1_X\backslash0}\langle -1\rangle$; the $\sigma$-action is $-1$ on the first term, and $1$ on the second. Over $0$, $i^*m_!m^*\1_{\AAA^1_X}\simeq \1_0[-1]\oplus \1_0^{\oplus 2}\langle -1\rangle$; the $\sigma$-action is $-1$ on the first term\footnote{This can be seen, for example, by looking at the exact triangle $\1_C\rightarrow\1_{C_1}\oplus \1_{C_2}\xrightarrow{r_1-r_2} \1_0\rightarrow$, where $C$ denotes $m^{-1}(0)$, $C_1$, $C_2$ are the two irreducible components, and $r_i$ is the unit map from the adjunction. Geometrically, the $-1$ comes from the fact that a non-trivial $1$-cycle of $C$ relative to the two infinity points is reversed when the two infinity points are switched. I thank Zhiwei Yun for teaching me this perspective.}, and switching factors on the second. Now consider the canonical map $m_!m^*\1_{\AAA^1_X}\simeq m_!m^! \1_{\AAA^1_X}\langle -1\rangle \rightarrow \1_{\AAA^1_X}\langle -1\rangle$. As $\sigma$ acts by $1$ on the right hand side, one sees that, over $\AAA^1_X\backslash0$ and $0$, the map is an isomorphism when restricted to $\sigma$-acting-by-$1$ parts of the left hand side (note that over $0$, $\1_0^{\oplus 2}\langle -1\rangle$ decomposes into a $1$-part and $(-1)$-part). It then follows that $\sigma$ must act by $-1$ on $q_!((m_!m^*\1_{\AAA_X^1})\otimes \CL)$, using Lemma \ref{lem_coho_L_zero}.
\end{proof}

By Remark \ref{rmk_sign_classical}, in the following, we will view $F'[-1]^*\Td$ as the right adjoint of $F$ with unit $(-1)^d\eta^{-1}_E: \id_E\isoto F'F[-1]^*\Td$ and counit $\eta_{E'}: FF'[-1]^*\Td\isoto\id_{E'}$, or the left adjoint of $F$ with unit $\eta^{-1}_{E'}$ and counit $(-1)^d\eta_E$. On the other hand, we can compute the right adjoint of $F$ formally:
\begin{align*}
p_*\RHOM(m^*\CL_1,p'^{!}(-))
&\simeq p_*((m^*\CL_1)^{\vee}\OX p'^{!}(-))\,\,(\text{$m^*\CL_1$ is dualisable, $(-)^{\vee}$ denotes $\RHOM((-),\1)$})\\
&\simeq p_*(m^*\CL_{p-1}\OX p'^{!}(-))\,\,(\text{$(-)^{\vee}$ commutes with pullback, and Lemma \ref{lem_L_dual}})\\
&\simeq p_*(m^*[-1]^*\CL_{1}\OX p'^{!}(-))\,\,(\text{Lemma \ref{lem__1_L}})\\
&\simeq [-1]^*p_*(m^*\CL_{1}\OX p'^{!}(-))
\end{align*}

So we have a natural isomorphism $F'\simeq p_*(m^*\CL_{1}\OX p'^{!}(-))\Tmd$. As $p'$ is smooth, we get a natural isomorphism $F'\simeq p_*(m^*\CL_{1}\OX p'^{*}(-))$. We summarise this discussion:

\begin{lemma}[Right adjoint]\label{lem_miracle}
    There are natural isomorphisms $F\simeq p'_*(p^{!}(-)\OX m^*\CL)\Tmd\simeq p'_*(p^*(-)\OX m^*\CL)$.
\end{lemma}

\begin{lemma}[Verdier duality]
    There is a natural isomorphism $\DD_{E'/X}F\simeq (F\DD_{E/X})[-1]^*\Td$.
\end{lemma}
\begin{proof}
    We compute, for every $A\in D(E)$:
\begin{align*}
\DD_{E'/X}FA
&=\RHOM(p'_!(p^*A\OX m^*\CL_1), \pi_{E'}^!\1_X)\\
&\simeq p'_*\RHOM(p^*A\OX m^*\CL_1, p'^!\pi_{E'}^!\1_X)\\
&\simeq p'_*\RHOM(m^*\CL_1, \RHOM(p^*A,p'^!\pi_{E'}^!\1_X))\\
&\simeq p'_*((m^*\CL_1)^{\vee}\OX \RHOM(p^*A,p^!\pi_E^!\1_X))\\
&\simeq [-1]^*p'_*(m^*\CL_1\OX p^!\DD_{E/X}A)\simeq (F\DD_{E/X}A)[-1]^*\Td
\end{align*}
where in the last step we used the first isomorphism in Lemma \ref{lem_miracle}.
\end{proof}

\begin{lemma}[Base change]\label{lem_F_base_change}
    Let $f: \widetilde{X}\rightarrow X$ be a map locally of finite type of finite dimensional schemes over $k$, and $E\rightarrow X$ be a vector bundle. Consider the following diagram, where each square is Cartesian:
\[\begin{tikzcd}
	& {\widetilde{E}\times_X\widetilde{E}'} && {E\times_XE'} & \\
	&& {\widetilde{E}'} && {E'} \\
	{\widetilde{E}} && E \\
	& {\widetilde{X}} && X
	\arrow[from=1-2, to=1-4]
	\arrow[from=1-2, to=2-3]
	\arrow[from=1-2, to=3-1]
	\arrow[dotted, from=1-2, to=4-2]
	\arrow[from=1-4, to=2-5]
	\arrow[from=1-4, to=3-3]
	\arrow[dotted, from=1-4, to=4-4]
	\arrow["{f}"{pos=0.7}, dotted, from=2-3, to=2-5]
	\arrow[dotted, from=2-3, to=4-2]
	\arrow[from=2-5, to=4-4]
	\arrow["{f}"{pos=0.7}, from=3-1, to=3-3]
	\arrow[from=3-1, to=4-2]
	\arrow[from=3-3, to=4-4]
	\arrow["f", from=4-2, to=4-4]
\end{tikzcd}\]
There are natural isomorphisms:
$$ (1)\,\, Ff_!\simeq f_!\tF,\,\,\,\, (2)\,\, f^*F\simeq \tF f^*,\,\,\,\, (3)\,\, f^!F\simeq \tF f^!,\,\,\,\, (4)\,\, Ff_*\simeq f_*\tF.$$
\end{lemma}

\begin{proof}
    (1) and (2) are formal consequences of proper base change and the projection formula, using the above diagram. (3) and (4) are obtained from (1) and (2), respectively, by taking right adjoints and relabelling.
\end{proof}


\begin{lemma}[Functoriality]\label{lem_F_functoriality}
    Let $\alpha: \tE\rightarrow E$ be a map of vector bundles over a finite dimensional scheme $X$ over $k$. Consider the following diagram:
\[\begin{tikzcd}[row sep=0.8cm, column sep=0.5cm]
	& {\widetilde{E}'} && {E'} & \\
	{\widetilde{E}\times_X\widetilde{E}'} && {\widetilde{E}\times_X E'} && {E\times_X E'} \\
	& {\widetilde{E}} && E
	\arrow["{\alpha'}"', from=1-4, to=1-2]
	\arrow[from=2-1, to=1-2]
	\arrow[from=2-1, to=3-2]
	\arrow["\square"{description, pos=0.4}, draw=none, from=2-3, to=1-2]
	\arrow[from=2-3, to=1-4]
	\arrow[from=2-3, to=2-1]
	\arrow[from=2-3, to=2-5]
	\arrow[from=2-3, to=3-2]
	\arrow["\square"{description}, draw=none, from=2-3, to=3-4]
	\arrow[from=2-5, to=1-4]
	\arrow[from=2-5, to=3-4]
	\arrow["\alpha"', from=3-2, to=3-4]
\end{tikzcd}\]
There are natural isomorphisms:
    $$ (1)\,\, F\alpha_!\simeq\alpha'^*\tF,\,\,\,\, (2)\,\, \alpha'_!F\Td\simeq \tF\alpha^*\Tdt, \,\,\,\, (3)\,\, F\alpha_*\Td\simeq\alpha'^!\tF\Tdt,\,\,\,\, (4)\,\, \alpha'_*F\simeq\tF\alpha^!.$$
\end{lemma}

\begin{proof}
    (1) is a formal consequence of proper base change and the projection formula, using the above diagram. Take right adjoints and relabel, we get (3). Compose (1) with $F'$ on the left and $\tF'$ on the right, and use Involutivity, we get (2$'$) $[-1]^*\Tmd\alpha_!\tF'\simeq F'\alpha'^* [-1]^*\Tmdt$. Take right adjoints, we get (4). Apply (2$'$) to $\alpha'$, we get (2).
\end{proof}

The above construction of the natural transformations makes the following Lemma apparent:

\begin{lemma}[Compatibility with units and counits]\label{lem_BC}
    The set-up is as in Lemma \ref{lem_F_functoriality}. Then we have four commutative diagrams
\[\begin{tikzcd}
	& {\tF\alpha^!\alpha_!} && {F\alpha_!\alpha^!} && {F\alpha_*\alpha^*} && {\tF\alpha^*\alpha_*} \\
	\tF && F && F && \tF \\
	& {\alpha'_*\alpha'^*\tF} && {\alpha'^*\alpha'_*F} && {\alpha'^!\alpha'_!F} && {\alpha'_!\alpha'^!\tF}
	\arrow["\simd", tail reversed, from=1-2, to=3-2]
	\arrow["{\mathrm{counit}}"', from=1-4, to=2-3]
	\arrow["\simd", tail reversed, from=1-6, to=3-6]
	\arrow["{\mathrm{counit}}"', from=1-8, to=2-7]
	\arrow["\simd", tail reversed, from=1-8, to=3-8]
	\arrow["{\mathrm{unit}}", from=2-1, to=1-2]
	\arrow["{\mathrm{unit}}"', from=2-1, to=3-2]
	\arrow["{\mathrm{unit}}", from=2-5, to=1-6]
	\arrow["{\mathrm{unit}}"', from=2-5, to=3-6]
	\arrow["\simd"', tail reversed, from=3-4, to=1-4]
	\arrow["{\mathrm{counit}}", from=3-4, to=2-3]
	\arrow["{\mathrm{counit}}", from=3-8, to=2-7]
\end{tikzcd}\]
where the vertical isomorphisms are from Functoriality.
\end{lemma}
\begin{proof}
    We sketch the proof of the leftmost one, the others are similar. In Lemma \ref{lem_F_functoriality}, (2$'$) is the left Beck–Chevalley transformation of (1) with respect to the following diagram (a reference for this notion is \cite[\nopp §2.4.1]{GL}), viewing $[-1]^*\Tdt\tF'$ (resp. $[-1]^*\Td F'$) as the left adjoint of $\tF$ (resp. $F$):
\[\begin{tikzcd}
	{D(\tE)} & {D(\tE')} \\
	{D(E)} & {D(E')}
	\arrow["\tF", from=1-1, to=1-2]
	\arrow["{\alpha_!}"', from=1-1, to=2-1]
	\arrow["{\alpha'^*}", from=1-2, to=2-2]
	\arrow["F"', from=2-1, to=2-2]
\end{tikzcd}\]
One can then check that (4) coincides with the right Beck–Chevalley transformation of (1) with respect to the transpose of the above diagram (\textit{cf.} \cite[Remark 2.4.1.3]{GL}). The conclusion then follows from the fact that the Beck–Chevalley transformations are compatible with units and counits.
\end{proof}

\begin{lemma}[Convolution]
    Let $(-)\ast(-): D(E)\X D(E)\rightarrow D(E)$ be the convolution with respect to the additive structure of $E$, similarly for $E'$. For $A,B\in D(E)$, we have natural isomorphisms:
    $$F(A\ast B)\simeq FA\OX FB\text{ and }F(A\OX B)\simeq (FA\ast FB)\Td.$$
\end{lemma}
\begin{proof}
    The first formula is a formal consequence of Functoriality (1), the additivity of the Artin–Schreier motive (Proposition \ref{prop_L_additive}) and the Künneth formula (Lemma \ref{lem_kunneth}). The second follows from the first, using Involutivity.
\end{proof}

\begin{lemma}[Plancherel]
    For $A,B\in D(E)$, we have a natural isomorphism:
    $$\pi_{E'!}(FA\OX FB)\simeq\pi_{E!}(A\OX[-1]^*B)\Tmd.$$
\end{lemma}
\begin{proof}
    Using Convolution, we compute: $\pi_{E'!}(FA\OX FB)\simeq\pi_{E'!}F(A\ast B)\simeq \iota_{E}^* (A\ast B)\Tmd\simeq \pi_{E!}(A\OX[-1]^*B)\Tmd$, where in the second step we used Functoriality (2).
\end{proof}

For a vector bundle $E$ over a scheme $X$, for every $n\in \ZZ$, consider the weight-$n$ homothety action $\Gm\X E\rightarrow E, (\lambda,v)\mapsto \lambda^n v$. We denote $\Gm$ by $\Gm(n)$ when considering the $n$-th homothety action for $n\ne 1$. Denote by $\rho_n: E\rightarrow E/\Gm(n)$ the projection to the quotient stack.

\begin{definition}\label{def_mon_obj}
    The category $D_{\mathrm{mon}}(E)$ of monodromic motives on $E$ is the thick subcategory of $D(E)$ generated by $\rho^*_n(D(E/\Gm(n)))$, as $n$ ranges over all non-zero integers.
\end{definition}

In other words, every monodromic motive on $E$ can be obtained from objects in $\rho^*_n(D(E/\Gm(n)))$ for various $n$ by finitely many operations of shifts, taking cofibres and direct summands. Note that we have a canonical identification $E/\Gm(n)\simeq E/\Gm(-n)$ induced by the inversion automorphism of $\Gm$. Consequently, in the above definition, we can equivalently let $n$ range over all positive integers. The following diagram will be used repeatedly below:
\[\begin{tikzcd}
	& {E\X_XE'} && {\AAA^1} \\
	E & {\frac{E\X_XE'}{\Gm(n)}} & {E'} & {\AAA^1} \\
	{E/\Gm(n)} & {E/\Gm(n)\X_XE'/\Gm(n)} & {E'/\Gm(n)} & {\AAA^1/\Gm(n)}
	\arrow["m"{description}, curve={height=-18pt}, from=1-2, to=1-4]
	\arrow["p"', from=1-2, to=2-1]
	\arrow[from=1-2, to=2-2]
	\arrow["{p'}", from=1-2, to=2-3]
	\arrow[equal, from=1-4, to=2-4]
	\arrow["\square"{description, pos=0.55}, draw=none, from=2-1, to=2-2]
	\arrow["{\rho_n}"', from=2-1, to=3-1]
	\arrow["{\bar{m}}"{description}, curve={height=-18pt}, from=2-2, to=2-4]
	\arrow["{\bar{p}}"', from=2-2, to=3-1]
	\arrow[from=2-2, to=3-2]
	\arrow["{\bar{p}'}", from=2-2, to=3-3]
	\arrow["\square"{description}, draw=none, from=2-3, to=2-2]
	\arrow["{\rho'_n}", from=2-3, to=3-3]
	\arrow["{q_n}", from=2-4, to=3-4]
	\arrow["\pi", from=3-2, to=3-1]
	\arrow["{\pi'}"', from=3-2, to=3-3]
	\arrow["\mu"{description}, curve={height=-18pt}, from=3-2, to=3-4]
\end{tikzcd}\]
Here the lower curved square is also Cartesian, $\Gm(n)$ acts on $E\X_XE'$ by weight $n$ (resp. $-n$) on $E$ (resp. $E'$), $E\X_XE'\rightarrow\frac{E\X_XE'}{\Gm(n)}$ is the quotient map, and $\frac{E\X_XE'}{\Gm(n)}\rightarrow E/\Gm(n)\X_XE'/\Gm(n)$ is the quotient map of the action induced by the weight $n$ action on $E$.

\begin{lemma}[Monodromicity]\label{lem_monodromicity}
    The Fourier transform restricts to an equivalence $D_{\mathrm{mon}}(E)\isoto D_{\mathrm{mon}}(E')$.
\end{lemma}
\begin{proof}
    The notations are as in the diagram above. By Involutivity, it suffices to show that monodromicity is preserved by $F$. For this, it suffices to write $p'_!(p^*\rho_n^*A\otimes m^*\CL)$ as a $\rho'^*_n$ pullback of an object on $E'/\Gm(n)$, for every $A\in D(E/\Gm(n))$. But $p'_!(p^*\rho_n^*A\otimes m^*\CL)\simeq \rho'^*_n\bar{p}'_!(\bar{p}^*A\otimes \bar{m}^*\CL)$ by standard diagram chase, using proper base change and the projection formula.
\end{proof}

\begin{definition}[{\cite{laumon_homogene}}]\label{def_F_1}
    The homogeneous Fourier transform is the integral transform $F_1: D(E/\Gm)\rightarrow D(E'/\Gm)$ with kernel $\beta_*\1[-1]\in D(\AAA^1/\Gm)$, where $\beta$ is the open immersion from $\Gm/\Gm(\simeq \Spec(k))$ to $\AAA^1/\Gm$. In other words, $F_1A:= \pi'_!(\pi^*A\otimes \mu^*(\beta_*\1[-1]))$.
\end{definition}

\begin{lemma}[Compatibility with the homogeneous Fourier transform]
    We have a commutative diagram:
\[\begin{tikzcd}
	{D(E)} & {D(E')} \\
	{D(E/\Gm)} & {D(E'/\Gm)}
	\arrow["F", from=1-1, to=1-2]
	\arrow["{\rho_1^*}", from=2-1, to=1-1]
	\arrow["{F_1}"', from=2-1, to=2-2]
	\arrow["{\rho'^*_1}"', from=2-2, to=1-2]
\end{tikzcd}\]
\end{lemma}
\begin{proof}
    Let $A\in D(E/\Gm)$. As in Lemma \ref{lem_monodromicity}, standard diagram chase gives $p'_!(p^*\rho_1^*A\otimes m^*\CL)\simeq \rho'^*_1\bar{p}'_!(\bar{p}^*A\otimes \bar{m}^*\CL)\simeq \rho'^*_1\pi'_!(\pi^*A\otimes\mu^* q_{1!}\CL)$. So it suffices to provide an isomorphism $q_{1!}\CL\simeq \beta_*\1[-1]$. For the rest of this proof, we abbreviate $q_1$ by $q$. Consider the following diagram:
\[\begin{tikzcd}
	0 & {\AAA^1} & \Gm \\
	{B\Gm} & {\AAA^1/\Gm} & {\Spec(k)}
	\arrow["i", hook, from=1-1, to=1-2]
	\arrow["q"', from=1-1, to=2-1]
	\arrow["\square"{description, pos=0.6}, draw=none, from=1-1, to=2-2]
	\arrow["q", from=1-2, to=2-2]
	\arrow["\square"{description, pos=0.6}, draw=none, from=1-2, to=2-3]
	\arrow["j"', from=1-3, to=1-2]
	\arrow["q", from=1-3, to=2-3]
	\arrow["\alpha"', hook, from=2-1, to=2-2]
	\arrow["\beta", from=2-3, to=2-2]
\end{tikzcd}\]
We have an exact triangle $\alpha_*\alpha^!q_!\CL\rightarrow q_!\CL\rightarrow\beta_*\beta^*q_!\CL\rightarrow$. By proper base change, Lemma \ref{lem_coho_L_zero} and using the exact triangle $j_!j^*\CL\rightarrow \CL\rightarrow i_*i^*\CL\rightarrow$, we get $\beta_*\beta^*q_!\CL\simeq \beta_*\1[-1]$. It remains to show $\alpha^!q_!\CL=0$. We follow the method in \cite[Remarque 2.3.1]{laumon_homogene}. We have the following diagram
\[\begin{tikzcd}
	0 & {\AAA^1} \\
	{B\Gm} & {\AAA^1/\Gm}
	\arrow["i"', curve={height=12pt}, hook, from=1-1, to=1-2]
	\arrow["q"', from=1-1, to=2-1]
	\arrow["r"', from=1-2, to=1-1]
	\arrow["q", from=1-2, to=2-2]
	\arrow["\alpha"', curve={height=12pt}, hook, from=2-1, to=2-2]
	\arrow["r"', from=2-2, to=2-1]
\end{tikzcd}\]
where the horizontal arrows $r$ are projections, note both squares are Cartesian. We have a canonical map $\alpha^!\rightarrow r_!$ defined by $\alpha^!\xrightarrow{\mathrm{unit}}\alpha^! r^!r_!\simeq r_!$. We claim that $\alpha^!q_!\CL\rightarrow r_!q_!\CL$ is an isomorphism. This will complete the proof, as $r_!q_!\CL\simeq q_!r_!\CL=0$, by Lemma \ref{lem_coho_L_zero}. For the claim, as $q$ is a smooth surjection, it suffices to show $q^!\alpha^!q_!\CL\isoto q^!r_!q_!\CL$ (using $q^!\simeq q^*\langle 1\rangle$), which by proper base change is equivalent to $i^!q^!q_!\CL\isoto r_!q^!q_!\CL$. As $q^!q_!\CL$ is $\Gm$-equivariant, this follows from the Contraction Principle (\cite[Appendix B]{FYZ}, \cite[\nopp §8.3]{FK})\footnote{One might be able to show directly that $q_!\CL$ satisfies the Contraction Principle, thus avoiding the $q^!$ step.}.
\end{proof}

\begin{lemma}[Constructibility]
    The Fourier transform restricts to an equivalence between constructible objects: $D_{\mathrm{c}}(E)\isoto D_{\mathrm{c}}(E')$. 
\end{lemma}
\begin{proof}
    By Involutivity, it suffices to show that constructibility is preserved by $F$. This is immediate, as all functors involved preserve constructibility.
\end{proof}

\section{The motivic Fourier transform on derived vector bundles}\label{sec_F_dVect}
We now construct the theory of Fourier transform for \'etale $L$-motives on derived vector bundles over derived Artin stacks. Such a theory has been constructed for $\ell$-adic sheaves in \cite{FYZ}, and for $\Gm$-equivariant \'etale motives in \cite{khan_homogeneous} (\textit{cf.} \cite[\nopp §8]{FK}). We follow the method of \cite[Appendix A]{FYZ}. See Conventions for the notations.

\subsection{Basic facts on derived vector bundles}\label{subsec_dVect}
A derived vector bundle $E$ on a derived Artin stack $X$ is the total space associated to a perfect complex $\CE\in\mathrm{Perf}(X)$. The total space functor $\mathrm{Tot}: \mathrm{Perf}(X)\rightarrow \mathrm{dArtSt}_k$ preserves Cartesian squares.  There are canonical maps $\iota_{E}: X\rightarrow E$ (zero section) and $\pi_E: E\rightarrow X$ (projection). If $\CE$ has tor-amplitude in $[0,+\infty)$ (resp. $[1,+\infty)$), then $\iota_E$ is a closed immersion and $\pi_E$ is representable by derived affine schemes (resp. a closed immersion); if $\CE$ has tor-amplitude in $(-\infty,0]$ (resp. $(-\infty,-1]$), then $\iota_E$ is quasi-smooth (resp. smooth) and $\pi_E$ is smooth (\cite[Lemma 6.1.5]{FYZ}).\\

By a \underline{global presentation} (or simply \underline{presentation}) of $\CE$ (or $E$) we mean an isomorphism $\CE\cong( \cdots\rightarrow\CE^{-1}\rightarrow\CE^0\rightarrow\CE^1\rightarrow\cdots)$, where the right hand side is (the object in $D(\CO_X)$ associated to) a complex of locally free sheaves of finite rank with finitely many non-zero terms. By a \underline{(global) presentation} of a map $\alpha: \widetilde{\CE}\rightarrow\CE$ we mean an isomorphism of $\alpha$ with a map induced by a map of complexes of some presentations of $\widetilde{\CE}$ and $\CE$. According to \cite[\nopp §2.2]{FK}, presentations always exist smooth locally. Given a presentation, we denote by $\CE^{\leq0}$ the complex $(\cdots\rightarrow\CE^{-1}\rightarrow\CE^0)$, and by $\CE^{\geq0}$ the complex $(\CE^0\rightarrow\CE^1\rightarrow\cdots)$. We have a Cartesian square:
\[\begin{tikzcd}
	{E^{\geq0}} & {E^0} \\
	E & {E^{\leq0}}
	\arrow[from=1-1, to=1-2]
	\arrow[from=1-1, to=2-1]
	\arrow["\square"{description, pos=0.55}, draw=none, from=1-1, to=2-2]
	\arrow["p", from=1-2, to=2-2]
	\arrow["i"', from=2-1, to=2-2]
\end{tikzcd}\]
Here $i$ is a closed immersion and $p$ is smooth (\cite[Lemma 6.1.5]{FYZ}).

\begin{lemma}\label{lem_bootstrap}
    Let $p: \tE\rightarrow E$ be a smooth map of derived vector bundles over a derived Artin stack $X$. Then $p$ is surjective and the adjunction maps $\id\rightarrow p_*p^*$, $p_!p^!\rightarrow \id$ are isomorphisms.
\end{lemma}
\begin{proof}
    \underline{Surjectivity}: Denote by $C$ the total space associated to the cofibre of $\widetilde{\CE}\rightarrow \CE$. As $p$ is smooth, $C$ has tor-amplitude in $(-\infty,-1]$. We have the Cartesian square: 
\[\begin{tikzcd}
	\tE & E \\
	X & C
	\arrow["p", from=1-1, to=1-2]
	\arrow[from=1-1, to=2-1]
	\arrow["\square"{description}, draw=none, from=1-1, to=2-2]
	\arrow[from=1-2, to=2-2]
	\arrow["{\iota_C}"', from=2-1, to=2-2]
\end{tikzcd}\]
The surjectivity of $p$ reduces to the surjectivity of $\iota_C$. The question being smooth local on $X$, we may assume $C$ has a presentation $(\CCC^{-n}\rightarrow\cdots\rightarrow\CCC^{-1})$. We do induction on $n$. We have 
\[\begin{tikzcd}
	X & {C'} & C \\
	& X & {C^{-n}}
	\arrow["{\iota_{C'}}"', from=1-1, to=1-2]
	\arrow["{\iota_C}", curve={height=-18pt}, from=1-1, to=1-3]
	\arrow[from=1-2, to=1-3]
	\arrow[from=1-2, to=2-2]
	\arrow["\square"{description, pos=0.55}, draw=none, from=1-2, to=2-3]
	\arrow[from=1-3, to=2-3]
	\arrow["{\iota_{C^{-n}}}"', from=2-2, to=2-3]
\end{tikzcd}\]
where $C'$ is the total space associated to $\fib(\CCC\rightarrow\CCC^{-n})[-1]$; it has tor-amplitude in $[-n+1,-1]$. It suffices to show the surjectivity of $\iota_{C'}$ and $\iota_{C^{-n}}$. The former follows from the induction hypothesis; the latter (and the base case $n=1$) is clear because $C^{-n}$ is the $n$-th classifying stack of some group over $X$.\\\\
\underline{Isomorphism}: The second isomorphism follows from the first because $p^!\simeq p^*\langle \dim p\rangle$ and the first (resp. second) isomorphism is equivalent to $p^*$ (resp. $p^!$) being fully faithful. We now prove $\id\isoto p_*p^*$. By the above, $p$ is smooth and surjective. Consider its associated Čech nerve:
\[\begin{tikzcd}
	\cdots & {G\X_{\tE}G} & G & \tE & E
	\arrow[shift right, from=1-1, to=1-2]
	\arrow[shift left, from=1-1, to=1-2]
	\arrow[shift right=3, from=1-1, to=1-2]
	\arrow[shift left=3, from=1-1, to=1-2]
	\arrow[from=1-2, to=1-3]
	\arrow[shift left=2, from=1-2, to=1-3]
	\arrow[shift right=2, from=1-2, to=1-3]
	\arrow[shift right, from=1-3, to=1-4]
	\arrow[shift left, from=1-3, to=1-4]
	\arrow["{a_0(=p)}", from=1-4, to=1-5]
\end{tikzcd}\]
Here $G= \tE\X_{E}\tE$. It is a derived vector bundle over $\tE$ with tor-amplitude in $(-\infty,0]$. Indeed, this is clear from the following diagram, noting that $\tE\rightarrow C$ factors as $\tE\xrightarrow{\pi_{\tE}} X\xrightarrow{\iota_C}C$:
\[\begin{tikzcd}
	G & \tE \\
	\tE & E \\
	X & C
	\arrow[from=1-1, to=1-2]
	\arrow[from=1-1, to=2-1]
	\arrow["\square"{description}, draw=none, from=1-1, to=2-2]
	\arrow[from=1-2, to=2-2]
	\arrow[from=2-1, to=2-2]
	\arrow[from=2-1, to=3-1]
	\arrow["\square"{description}, draw=none, from=2-1, to=3-2]
	\arrow[from=2-2, to=3-2]
	\arrow[from=3-1, to=3-2]
\end{tikzcd}\]
Denote the map $G\X_{\tE}G\X_{\tE}\cdots\X_{\tE}G\text{ ($n$-times)}\rightarrow E$ by $a_n$, and the structural map $G\X_{\tE}G\X_{\tE}\cdots\X_{\tE}G\rightarrow \tE$ by $q_n$. For every $A\in D(E)$, we compute: $A\simeq\lim_{\Delta}a_{n*}a_n^*A\simeq\lim_{\Delta}p_*q_{n*}q_n^*p^*A\simeq\lim_{\Delta}p_*p^*A\simeq p_*p^*A$. Here the first step follows from the construction of $D(E)$, and the third step is the content of Lemma \ref{lem_blob}.1 below.
\end{proof}

\begin{lemma}\label{lem_blob}
    Let $E$ be a derived vector bundle over a derived Artin stack $X$.\\
    (1) (\textit{cf.} \cite[Proposition A.10]{khan_virtual}) If $E$ has tor-amplitude in $(-\infty,0]$, then the adjunction maps $\id\rightarrow \pi_{E*}\pi_E^*$ and $\pi_{E!}\pi_E^!\rightarrow\id$ are isomorphisms.\\
    (2) If $E$ has tor-amplitude in $(-\infty,-1]$, then the adjunction map $\pi_E^*\pi_{E*}\rightarrow\id$ is an isomorphism. Furthermore, we have $\pi_{E}^!\simeq \pi_{E}^*\Td$ and $\pi_{E!}\simeq\pi_{E*}\Tmd$.
\end{lemma}
\begin{proof}
    (1) As in the proof of Isomorphism above, it suffices to prove the first isomorphism. The question being smooth local (by smooth base change), we may assume the perfect complex $\CE$ has a presentation $(\CE^{-n}\rightarrow\cdots\rightarrow\CE^{-1}\rightarrow\CE^0)$. We do induction on $n$. Consider the canonical map $\CE^0\rightarrow\CE$, take the cofibre and pass to total spaces, we get the Cartesian square in the following diagram, where $F$ has tor-amplitude in $[-n,-1]$:
\[\begin{tikzcd}
	{E^0} & E \\
	X & F \\
	& X
	\arrow[from=1-1, to=1-2]
	\arrow["{\pi_{E^0}}"', from=1-1, to=2-1]
	\arrow["\square"{description}, draw=none, from=1-1, to=2-2]
	\arrow["\alpha", from=1-2, to=2-2]
	\arrow["{\iota_F}"', from=2-1, to=2-2]
	\arrow["{\pi_F}", from=2-2, to=3-2]
\end{tikzcd}\]
The statement for $\pi_E$ then reduces to those for $\alpha$ and $\pi_F$. As $\iota_F$ is smooth and surjective, the statement for $\alpha$ reduces to that for $\pi_{E^0}$, which follows from $\AAA^1$-invariance. For $\pi_F$, consider the diagram:
\[\begin{tikzcd}
	{F'} & X \\
	X & F \\
	& X
	\arrow[from=1-1, to=1-2]
	\arrow["{\pi_{F'}}"', from=1-1, to=2-1]
	\arrow["\square"{description}, draw=none, from=1-1, to=2-2]
	\arrow["{\iota_F}", from=1-2, to=2-2]
	\arrow["{\iota_F}"', from=2-1, to=2-2]
	\arrow["{\pi_F}", from=2-2, to=3-2]
\end{tikzcd}\]
Note $F'$ has tor-amplitude in $[-n+1,0]$. As $\pi_F\circ\iota_F=\id$, the statement for $\pi_F$ reduces to that for $\iota_F$, then further to that for $\pi_{F'}$, which is known by the induction hypothesis.\\

(2) We first show $\pi_E^*\pi_{E*}\isoto \id$. Consider the following diagram:
\[\begin{tikzcd}
	\cdots & {G\X_XG} & G & X & E & X
	\arrow[shift right, from=1-1, to=1-2]
	\arrow[shift left, from=1-1, to=1-2]
	\arrow[shift right=3, from=1-1, to=1-2]
	\arrow[shift left=3, from=1-1, to=1-2]
	\arrow[from=1-2, to=1-3]
	\arrow[shift left=2, from=1-2, to=1-3]
	\arrow[shift right=2, from=1-2, to=1-3]
	\arrow[shift right, from=1-3, to=1-4]
	\arrow[shift left, from=1-3, to=1-4]
	\arrow["{\iota_E(=a_0)}", from=1-4, to=1-5]
	\arrow["{\pi_E}", from=1-5, to=1-6]
\end{tikzcd}\]
Here the terms to the left of $E$ form the Čech nerve associated to $\iota_E$. Note $G=X\X_EX$ is the total space associated to $\CE[-1]$, hence a derived vector bundle over $X$ with tor-amplitude in $(-\infty,0]$. Denote the map $G\X_XG\X_X\cdots\X_XG\text{ ($n$-times)}\rightarrow E$ by $a_n$, and the structural map $G\X_XG\X_X\cdots\X_XG\rightarrow X$ by $q_n$. For every $A\in D(E)$, we compute:
\begin{align*}
\pi_{E*}A
&\simeq\pi_{E*}\lim_{\Delta}a_{n,*}a_n^*A\,\,(\text{$D(-)$ is a sheaf in the smooth topology})\\
&\simeq\lim_{\Delta}\pi_{E*}a_{n,*}a_n^*A\\
&\simeq\lim_{\Delta}q_{n*}q^*_n\iota_E^*A\,\,(\text{using $\pi_E\circ \iota_E=\id$})\\
&\simeq \iota_E^*A\,\,(\text{$\id\isoto q_{n*}q_n^*$ by (1)})
\end{align*}
We thus get an isomorphism $\pi_{E*}\simeq \iota_E^*$. As $\iota_E^*\simeq \iota_E^!\Td$ ($d$ being the rank of $E$), take left adjoints, we get $\pi_E^*\simeq \iota_{E!}\Tmd$, hence $\pi_E^*\pi_{E*}\simeq \iota_{E!}\iota_E^!$. It suffices to show $\iota_{E!}\iota_E^!\isoto\id$. Consider the following diagram:
\[\begin{tikzcd}
	G & X \\
	X & E
	\arrow[from=1-1, to=1-2]
	\arrow[from=1-1, to=2-1]
	\arrow["\square"{description}, draw=none, from=1-1, to=2-2]
	\arrow["{a_0}", from=1-2, to=2-2]
	\arrow["{a_0}"', from=2-1, to=2-2]
\end{tikzcd}\]
As $\iota_E$ is smooth and surjective, the question reduces to the corresponding statement for $G\rightarrow X$, which is known by (1).\\\\
Combine this with (1), we see that $\pi_E^*$ is an equivalence. The “Furthermore” part then follows from Poincaré duality (note $\pi_E$ is smooth).
\end{proof}

\subsection{Definition and involutivity}
Let $X$ be a derived Artin stack, and $E$ be a derived vector bundle on $X$. The definition of the Fourier transform $F: D(E)\rightarrow D(E')$ is identical to Definition \ref{def_fourier_transform}. Here the dual $E'$ of $E$ is by definition the total space associated to $\CE':=\mathrm{R\underline{Hom}}_{D(\CO_X)}(\CE,\CO_X)$, the dual of $\CE$, and the map $m$ is the composition $E\X_XE'\rightarrow \AAA^1_X\rightarrow\AAA^1$, where $E\X_XE'\rightarrow \AAA^1_X$ is induced by the canonical map $\CE\OX\CE'\rightarrow\CO_X$. We want to establish similar properties for $F$ as in §\ref{sec_fourier_on_vect}. First note that the construction of $F'F\simeq [-1]^*\Tmd$ in Proposition \ref{prop_involutivity} works only when the zero section $\iota_E: X\rightarrow E$ is a closed immersion, which is the case if and only if $\CE$ has tor-amplitude in $[0,+\infty)$. We follow the method of \cite[\nopp §A.3.1]{FYZ} to overcome this difficulty. 

\begin{remark}\label{rmk_base_func_still_hold}
    Lemmas \ref{lem_F_base_change}.1, \ref{lem_F_base_change}.2, \ref{lem_F_functoriality}.1 still hold, being formal consequences of proper base change and the projection formula.
\end{remark}

\begin{lemma}[{\cite[\nopp §A.2.2]{FYZ}}]\label{lem_base_change_smooth}
    Let $\alpha: \tE\rightarrow E$ be a map of derived vector bundles over a derived Artin stack $X$. Assume $\alpha$ is either smooth or a closed immersion. Then we have a canonical isomorphism $\alpha'_!F\Td\simeq \tF\alpha^*\Tdt$.
\end{lemma}
\begin{proof}
    The proof in \textit{loc. cit.} applies \textit{mutatis mutandis} to our set-up. We sketch the $\alpha$ smooth case. Consider the diagram (as in Lemma \ref{lem_F_functoriality}):
\[\begin{tikzcd}[row sep=0.8cm, column sep=0.5cm]
	& {\widetilde{E}'} && {E'} & \\
	{\widetilde{E}\times_X\widetilde{E}'} && {\widetilde{E}\times_X E'} && {E\times_X E'} \\
	& {\widetilde{E}} && E
	\arrow["{\alpha'}"', hook', from=1-4, to=1-2]
	\arrow["{\tp'}", from=2-1, to=1-2]
	\arrow["\tp"', from=2-1, to=3-2]
	\arrow["\square"{description, pos=0.4}, draw=none, from=2-3, to=1-2]
	\arrow[from=2-3, to=1-4]
	\arrow["i"', hook', from=2-3, to=2-1]
	\arrow["q", two heads, from=2-3, to=2-5]
	\arrow[from=2-3, to=3-2]
	\arrow["\square"{description}, draw=none, from=2-3, to=3-4]
	\arrow["{p'}"', from=2-5, to=1-4]
	\arrow["p", from=2-5, to=3-4]
	\arrow["\alpha"', two heads, from=3-2, to=3-4]
\end{tikzcd}\]
Denote by $\td$ (resp. $d$) the rank of $\tE$ (resp. $E$). For $A\in D(E)$, we have $\alpha'_!FA=\alpha'_!p'_!(p^*A\otimes m^*\CL)\isoot \alpha'_!p'_!q_!q^*(p^*A\otimes m^*\CL)\langle \td-d\rangle\simeq\tp'_!i_*i^*(\tp^*\alpha^*A\otimes \widetilde{m}^*\CL)\langle \td-d\rangle\leftarrow \tF\alpha^*A\langle \td-d\rangle$, where the second step uses Lemma \ref{lem_bootstrap} and the third step follows from standard diagram chasing. To show the last arrow is an isomorphism, using Lemma \ref{lem_point_conservative}, it suffices to show $(\tp'_!(\tp^*\alpha^*A\otimes \widetilde{m}^*\CL))|_v=0$ for every point $v\rightarrow \tE'\backslash E'$.\\\\
Consider the following diagram:
\[\begin{tikzcd}
	\tE & {\tE\X_X\tE'} & {C_v} & {\AAA^1_v} \\
	E & {E\X_X\tE'} & {E_v} \\
	X & {\tE'} & v
	\arrow["\alpha"', from=1-1, to=2-1]
	\arrow["\square"{description, pos=0.6}, draw=none, from=1-1, to=2-2]
	\arrow["\tp"', from=1-2, to=1-1]
	\arrow[from=1-2, to=2-2]
	\arrow["\square"{description, pos=0.4}, draw=none, from=2-2, to=1-3]
	\arrow[from=1-3, to=1-2]
	\arrow["{\widetilde{m}_v}", from=1-3, to=1-4]
	\arrow["{\alpha_v}", from=1-3, to=2-3]
	\arrow["\pi_E"', from=2-1, to=3-1]
	\arrow["\square"{description}, draw=none, from=2-1, to=3-2]
	\arrow[from=2-2, to=2-1]
	\arrow[from=2-2, to=3-2]
	\arrow["\square"{description, pos=0.45}, draw=none, from=2-2, to=3-3]
	\arrow[from=2-3, to=2-2]
	\arrow[from=2-3, to=3-3]
	\arrow["\pi_{\tE'}", from=3-2, to=3-1]
	\arrow[from=3-3, to=3-2]
\end{tikzcd}\]
We need to show $\pi_{v!}(\alpha_v^*(A|_{E_v})\otimes \widetilde{m}_v^*\CL)=0$, where $\pi_v$ is $C_v\rightarrow v$. It suffices to show $(\alpha_{v!}(\alpha_v^*(A|_{E_v})\otimes \widetilde{m}_v^*\CL))_w\simeq ((A|_{E_v})\otimes\alpha_{v!}\widetilde{m}_v^*\CL)|_w=0$ for every point $w\rightarrow E_v$ (Lemma \ref{lem_point_conservative}). But $C_w:=C_v\X_{E_v}w\xrightarrow{\widetilde{m}_w}\AAA^1_w$ is a non-zero linear map (as $v$ lands in $\tE'\backslash E'$) and $C_w$ is a derived vector bundle over $w$ with tor-amplitude in $(-\infty,0]$ (as $\alpha$ is smooth), one shows that $\widetilde{m}_w$ is then necessarily smooth, so $\pi_{w!}\widetilde{m}_w^*\CL=0$ using Lemma \ref{lem_bootstrap} and Lemma \ref{lem_coho_L_zero}, where $\pi_{w}$ is $C_w\rightarrow w$. This completes the proof.
\end{proof}

\begin{definition}[{\cite[Definition A.2.5]{FYZ}}]\label{def_inv_datum}
    Let $E$ be a derived vector bundle of rank $d$ over a derived Artin stack $X$. An \underline{involutivity datum} on $E$ is a pair $(\eta_E,\eta_{E'})$ of natural isomorphisms $\eta_E: F'\circ F\isoto [-1]^*\Tmd$ and $\eta_{E'}: F\circ F'\isoto [-1]^*\Tmd$, such that the following two diagrams commute: 
\[\begin{tikzcd}
	F & {F\circ F'\circ F[-1]^*\Td} & F & {F'} & {F'\circ F\circ F'[-1]^*\Td} & {F'}
	\arrow["{(-1)^d}", curve={height=24pt}, from=1-1, to=1-3]
	\arrow["{F\circ \eta_{E}}"', from=1-2, to=1-1]
	\arrow["\sim", from=1-2, to=1-1]
	\arrow["{\eta_{E'}\circ F}", from=1-2, to=1-3]
	\arrow["\sim"', from=1-2, to=1-3]
	\arrow["{(-1)^d}", curve={height=24pt}, from=1-4, to=1-6]
	\arrow["{F'\circ\eta_{E'}}"', from=1-5, to=1-4]
	\arrow["\sim", from=1-5, to=1-4]
	\arrow["\sim"', from=1-5, to=1-6]
	\arrow["{\eta_{E}\circ F'}", from=1-5, to=1-6]
\end{tikzcd}\]
\end{definition}

\begin{remark}\label{rmk_inv_datum_basics}
    (1) By definition, an involutivity datum on $E$ is the same data as an involutivity datum on $E'$. Furthermore, in fact the commutativity of one diagram above implies the other: suppose the left diagram commutes, to show the commutativity of the right diagram, it suffices to show it after right composing with $F$, then the claim follows from the evident commutativity of the following diagram:
\[\begin{tikzcd}
	{F'\circ F} & {F'\circ F\circ F'\circ F[-1]^*\Td} & {F'\circ F} \\
	& {F'\circ F} & \id
	\arrow["{(-1)^d}", curve={height=12pt}, from=1-1, to=2-2]
	\arrow["\sim", from=1-2, to=1-1]
	\arrow["{F'\circ\eta_{E'}\circ F}"', from=1-2, to=1-1]
	\arrow["\sim"', from=1-2, to=1-3]
	\arrow["{\eta_{E}\circ F'\circ F}", from=1-2, to=1-3]
	\arrow["\simd"', from=1-2, to=2-2]
	\arrow["{F'\circ F\circ \eta_{E}}", from=1-2, to=2-2]
	\arrow["{\eta_E}", from=1-3, to=2-3]
	\arrow["\simd"', from=1-3, to=2-3]
	\arrow["{\eta_E}", from=2-2, to=2-3]
	\arrow["\sim"', from=2-2, to=2-3]
\end{tikzcd}\]
(2) An involutivity datum can be presented in two other equivalent ways: (i) as a pair $(\alpha_E, \alpha_{E'})$ of isomorphisms $\alpha_E: \iota_{E!}\1_X\isoto F'\1_{E'}\Td$ and $\alpha_{E'}: \iota_{E'!}\1_X\isoto F\1_{E}\Td$ satisfying a certain condition\footnote{Namely, the corresponding pair $(\eta_E, \eta_{E'})$ satisfies the sign condition in Definition \ref{def_inv_datum}.}; (ii) as a pair $(\alpha'_E, \alpha'_{E'})$ of isomorphisms $\alpha_E':\1_X\rightarrow\iota_E^!F'\1_{E'}\Td$ and $\alpha_{E'}':\1_X\rightarrow\iota_{E'}^!F\1_{E}\Td$ satisfying a certain condition. Indeed, for (i): in one direction, given $(\eta_E, \eta_{E'})$, apply $\eta_E$ to $\iota_{E!}$, we get $\alpha_E$, similarly for $\alpha_{E'}$; the other direction is shown in \cite[Lemma A.2.4]{FYZ} (\textit{cf.} Proposition \ref{prop_involutivity}). For (ii): note that $\alpha_E$ and $\alpha_{E'}$ are related by adjunction, and \cite[Lemma A.3.3]{FYZ} (proved in the same way in our context) shows $\alpha'_E$ is an isomorphism if $\alpha_E$ is, similarly for $\alpha_{E'}$ and $\alpha'_{E'}$.\\\\
(3) A crucial observation is that, on a connected $X$, an involutivity datum (if it exists) is unique up to a scalar. More precisely: consider the $(\alpha'_E, \alpha'_{E'})$ presentation of an involutivity datum as in (2), then, by Lemma \ref{lem_hom_of_unit_conn_comp}, the set of isomorphisms $\1_X\rightarrow\iota_E^!F'\1_{E'}\Td$ is $\mathrm{Aut}(\1_X)\simeq L^{\X}$; and $\alpha'_E$ and $\alpha'_{E'}$ determine each other by the sign condition in Definition \ref{def_inv_datum}.
\end{remark}

\begin{lemma}[Base change of an involutivity datum]\label{lem_inv_base_change}
    Given a pullback digram of derived vector bundles 
\[\begin{tikzcd}
	{\widetilde{E}} & E \\
	{\widetilde{X}} & X
	\arrow["f", from=1-1, to=1-2]
	\arrow[from=1-1, to=2-1]
	\arrow["\square"{description}, draw=none, from=1-1, to=2-2]
	\arrow[from=1-2, to=2-2]
	\arrow["f"', from=2-1, to=2-2]
\end{tikzcd}\]
An involutivity datum $(\eta_E,\eta_{E'})$ on $E$ canonically induces an involutivity datum $(f^*\eta_E,f^*\eta_{E'})$ on $\tE$.
\end{lemma}
\begin{proof}
    We describe the construction of $f^*\eta_E$, the other one is similar, and it is easy to verify that they form an involutivity datum. Apply $f^*$ to $\eta_E$ and use Remark \ref{rmk_base_func_still_hold}, we get $f^*F'F\simeq \tF'\tF f^*\isoto f^*[-1]^*\Tmd$, apply this to $\iota_{E!}\1_X$ and use proper base change, we get $\tF'\1_{\tE'}\isoto \iota_{\tE!}\1_{\widetilde{X}}\Tmd$. As discussed in Remark \ref{rmk_inv_datum_basics}.2, this induces an isomorphism $f^*\eta_E: \tF'\tF\isoto [-1]^*\Tmd$.
\end{proof}

\begin{example}[{\cite[\nopp §A.3.1]{FYZ}}]\label{ex_involutivity_datum_examples}
    (1) For a (classical) vector bundle $E$ over a derived Artin stack, Proposition \ref{prop_involutivity} (proved in the same way over a derived Artin stack) gives an involutivity datum $(\eta_{E},\eta_{E'})$ on $E$ (the sign condition is checked in Remark \ref{rmk_sign_classical}). Note that this construction commutes with base change, \textit{i.e.}, for every base change $f: \tE\rightarrow E$ as above, $(f^*\eta_E,f^*\eta_{E'})\simeq(\eta_{\tE},\eta_{\tE'})$.\\\\
    (2) For a derived vector bundle $E$ over a point $X$, we construct an ``obvious'' involutivity datum on $E$. The perfect complex $\CE$ has a canonical presentation $\CE\simeq\oplus_{i\in\ZZ}(\CH^i\CE)[-i]$. The Fourier transform preserves box products (by the Künneth formula and the additivity of the Artin–Schreier motive), so it suffices to construct an involutivity datum on $E^i$, the total space of $(\CH^i\CE)[-i]$ (we are implicitly using the $(\alpha'_E, \alpha'_{E'})$ presentation of an involutivity datum).\\\\
    If $i=0$ this is given in (1). If $i>0$, then $X=E^i_{\mathrm{cl}}$ and we identify $D(X)$ and $D(E^i)$ by $\pi^*_{E^i}$ and $\pi_{E^i*}$. The dual $(E^i)'$ is the $i$-th classifying stack of a product of $\Ga$'s over $X$. By Lemma \ref{lem_blob}, $\pi_{(E^i)'}^*$ and $\pi_{(E^i)'*}$ are inverse to each other, and we have $\pi_{(E^i)'}^!\simeq \pi_{(E^i)'}^*\langle d^i\rangle$ and $\pi_{(E^i)'!}\simeq\pi_{(E^i)'*}\langle -d^i\rangle$. Identify $D(X)$ and $D((E^i)')$ by $\pi_{(E^i)'}^*$ and $\pi_{(E^i)'*}$. Then $F$ (resp. $F'$) is the endofunctor $\id$ (resp. $\langle -d^i\rangle$) on $D(X)$. We take $\eta_{E^i}=(-1)^{d^i}$ and $\eta_{(E^i)'}=\id$. The $i<0$ case is dual to this one.
\end{example}

\begin{lemma}[{\textit{cf.} \cite[Lemma A.3.1]{FYZ}}]\label{lem_small_open_inv_datum}
    Let $E$ be a derived vector bundle over a derived Artin stack $X$. Assume $\CE$ admits a presentation. Then:\\
    (1) There is a canonical involutivity datum on $E$.\\
    (2) The construction in (1) commutes with base change.\\
    (3) If $E$ is a classical vector bundle or if $X$ is a point, then this involutivity datum coincides with the one in Example \ref{ex_involutivity_datum_examples}.
\end{lemma}
See the beginning of §\ref{subsec_dVect} for the notion of a presentation and notations in the proof below.
\begin{proof}
    First choose a presentation of $\CE$. By Lemma \ref{lem_bootstrap}, we have a fully faithful imbedding $p^*i_!: D(E)\hookrightarrow D(E^0)$. Using Lemma \ref{lem_base_change_smooth}, we get $p^*i_!F'F[-1]^*\Td\simeq F'_0F_0p^*i_![-1]^*\langle d^0\rangle\xrightarrow[\sim]{\eta_{E^0}\circ p^*i_!} p^*i_!$, where $F$ (resp. $F_0$, resp. $F'$. resp. $F_0'$) denotes the Fourier transform on $E$ (resp. $E^0$, resp. $E'$, resp. $(E^0)'$), and $d$ (resp. $d^0$) is the rank of $E$ (resp. $E^0$). As $p^*i_!$ is fully faithful, we get an isomorphism $F'F[-1]^*\Td\isoto \id_E$. We denote this by $\eta_{E^0}|_E$ and call it the restriction of $\eta_{E^0}$ to $E$. It is easy to verify that $(\eta_{E^0}|_E,\eta_{(E^0)'}|_{E'})$ forms an involutivity datum.\\\\
    Given a base change diagram as in Lemma \ref{lem_inv_base_change}, the chosen presentation pulls back to a presentation of $\tE$. We use a “$\,\,\widetilde{}\,\,$” to denote the corresponding objects over $\widetilde{X}$. By Remark \ref{rmk_base_func_still_hold}, the pullback $(f^*\eta_{E^0},f^*\eta_{(E^0)'})$ (living on $\tE^0)$ restricts to an involutivity datum $((f^*\eta_{E^0})|_{\tE},(f^*\eta_{(E^0)'})|_{\tE'})$ on $\tE$. Note that $(f^*\eta_{E^0})|_{\tE}$ can be described in the $\alpha_{\tE}$ presentation as follows: by Remark \ref{rmk_base_func_still_hold} and proper base change, we have $f^*F'_0F_0p^*i_!\iota_{E!}\1_{X}\langle d^0\rangle\simeq \tF_0'\tF_0f^*p^*i_!\iota_{E!}\1_{X}\langle d^0\rangle\simeq\tF_0'\tF_0\widetilde{p}^*\widetilde{i}_!f^*\iota_{E!}\1_{X}\langle d^0\rangle\simeq\tF_0'\tF_0\widetilde{p}^*\widetilde{i}_!\iota_{\tE!}\1_{\widetilde{X}}\langle d^0\rangle$. By Lemma \ref{lem_base_change_smooth}, the last term is isomorphic to $\widetilde{p}^*\widetilde{i}_!\tF'\tF\iota_{\tE!}\1_{\widetilde{X}}\Td$, on the other hand it is isomorphic to $f^*p^*i_!\iota_{E!}\1_{X}\simeq\widetilde{p}^*\widetilde{i}_!\iota_{\tE!}\1_{\widetilde{X}}$ via $f^*\eta_{E^0}$, so we get $\widetilde{p}^*\widetilde{i}_!\tF'\tF\iota_{\tE!}\1_{\widetilde{X}}\Td\simeq\widetilde{p}^*\widetilde{i}_!\iota_{\tE!}\1_{\widetilde{X}}$. By the full faithfulness of $\widetilde{p}^*\widetilde{i}_!$, this gives an $\alpha_{\tE}: \tF'\1_{\widetilde{E}'}\Td\simeq\iota_{\tE!}\1_{\widetilde{X}}$, which corresponds to $(f^*\eta_{E^0})|_{\tE}$.\\\\
    Note that this description of $\alpha_{\tE}$ is the same as how one computes the $\alpha_{\tE}$ corresponding to $f^*(\eta_{E^0}|_E)$. We thus obtain: $((f^*\eta_{E^0})|_{\tE},(f^*\eta_{(E^0)'})|_{\tE'})\simeq(f^*(\eta_{E^0}|_E),f^*(\eta_{(E^0)'}|_{E'}))$. As the involutivity datum on $E^0$ commutes with pullbacks (Example \ref{ex_involutivity_datum_examples}.1), $(f^*(\eta_{E^0}|_E),f^*(\eta_{(E^0)'}|_{E'}))$ coincides with the involutivity datum constructed by applying the first paragraph to the pullback presentation on $\tE$.\\\\
    We can now prove that the involutivity datum on $E$ constructed as in the first paragraph is independent of the presentation of $E$. We may assume $X$ is connected. By Remark \ref{rmk_inv_datum_basics}.3, the involutivity datum is unique up to a scalar. By the above, this scalar can be computed after a base change. However, as shown in \cite[Lemma A.3.2]{FYZ} (proved in the same way here), if $X$ is a geometric point, the involutivity datum constructed as above coincides with the ``obvious'' one in Example \ref{ex_involutivity_datum_examples}.2. So the scalar is independent of the presentation. This concludes the proof.
\end{proof}

\begin{theorem}[Involutivity]\label{thm_involutivity}
    Every derived vector bundle $E$ over a derived Artin stack $X$ has a canonical involutivity datum.
\end{theorem}
We give a slightly different argument from that in \cite[Lemma A.3.4]{FYZ}, avoiding the use of $t$-structures on $D(X)$.
\begin{proof}
    Recall that $E$ admits a presentation smooth locally on $X$. For this proof, we call a derived Artin stack $U$ equipped with a smooth map to $X$ \underline{small} if $E$ admits a presentation on $U$. We will construct a map $\alpha'_{E}: \1_X\rightarrow\iota_E^!F'\1_{E'}\Td$ and show that, together with $\alpha'_{E'}$ obtained by the dual construction, they form a $(\alpha'_E,\alpha'_{E'})$ presentation of an involutivity datum.\\\\
    Denote $\iota_E^!F'\1_{E'}\Td$ by $\CK$. Lemma \ref{lem_small_open_inv_datum} gives a canonical isomorphism $\1_{U}\isoto\CK_U$ for every small $U$. Consider the $\Vect$-valued presheaf $\CF: V\mapsto \RHom_{D(V)}(\1_V,\CK_V)$ on the smooth site of $X$. This is in fact a sheaf, because $D(-)$ is a sheaf in the smooth topology. By Lemma \ref{lem_hom_of_unit_conn_comp}, $\CF(U)$ is coconnective for every small $U$, so $\CF$ is in fact a sheaf valued in $\Vect^{\geq0}$. It follows that $V\mapsto H^0\CF(V)=\Hom_{D(V)}(\1_V,\CK_V)$ is a sheaf valued in $\Vect^{\heartsuit}$.\footnote{This can be seen, for example, as follows: the local-to-global spectral sequence gives, for all $V$, $E^{p,q}_2=H^p(V,\CH^q\CF)\implies H^{p+q}(V,\CF)$. As $\CF$ is coconnective, we get $H^0(V,\CH^0\CF)\simeq H^0(V,\CF)$, \textit{i.e.}, sheafifying $V\mapsto H^0(V,\CF)$ does not change sections.} As the map $\1_{U}\rightarrow\CK_U$ on every connected small $U$ is unique up to a scalar, and the scalar commutes with restrictions (Lemma \ref{lem_small_open_inv_datum}), the maps on small $U$'s glue uniquely to a global map $\alpha'_E: \1_X\rightarrow\CK$. By looking at restrictions to small opens, it is clear that $\alpha'_E$, together with $\alpha'_{E'}$ obtained by the dual construction, form an involutivity datum.
\end{proof}

\subsection{Properties}\label{subsec_F_properties_d}
The discussion in §\ref{sec_fourier_on_vect} now applies \textit{mutatis mutandis} to the current set-up. Throughout this section, let $E$ be a derived vector bundle over a derived Artin stack $X$. See Conventions for the notations.
\begin{lemma}[Right adjoint]\label{lem_miracle_d}
    There is a natural isomorphism $F\simeq p'_*(p^{!}(-)\OX m^*\CL)\Tmd$. If $p$ is smooth (\textit{i.e.}, if $E$ has tor-amplitude in $[0,+\infty)$), then there is a further natural isomorphism with $p'_*(p^*(-)\OX m^*\CL)$.
\end{lemma}

\begin{lemma}[Verdier duality]
    There is a natural isomorphism $\DD_{E'/X}F\simeq (F\DD_{E/X})[-1]^*\Td$.
\end{lemma}

\begin{lemma}[Base change]\label{lem_F_base_change_d}
    Let $f: \widetilde{X}\rightarrow X$ be a map of derived Artin stacks, let $\tE= E\X_X\widetilde{X}$ be the pullback. There are natural isomorphisms:
$$ (1)\,\, Ff_!\simeq f_!\tF,\,\,\,\, (2)\,\, f^*F\simeq \tF f^*,\,\,\,\, (3)\,\, f^!F\simeq \tF f^!,\,\,\,\, (4)\,\, Ff_*\simeq f_*\tF.$$
\end{lemma}

\begin{lemma}[Functoriality]\label{lem_F_functoriality_d}
    Let $\alpha: \tE\rightarrow E$ be a map of derived vector bundles over $X$, and $\alpha'$ be the dual map. There are natural isomorphisms:
    $$ (1)\,\, F\alpha_!\simeq\alpha'^*\tF,\,\,\,\, (2)\,\, \alpha'_!F\Td\simeq \tF\alpha^*\Tdt, \,\,\,\, (3)\,\, F\alpha_*\Td\simeq\alpha'^!\tF\Tdt,\,\,\,\, (4)\,\, \alpha'_*F\simeq\tF\alpha^!.$$
\end{lemma}

\begin{lemma}[Compatibility with units and counits]\label{lem_BC_d}
    The set-up is as in Lemma \ref{lem_F_functoriality_d}. Then we have four commutative diagrams
\[\begin{tikzcd}
	& {\tF\alpha^!\alpha_!} && {F\alpha_!\alpha^!} && {F\alpha_*\alpha^*} && {\tF\alpha^*\alpha_*} \\
	\tF && F && F && \tF \\
	& {\alpha'_*\alpha'^*\tF} && {\alpha'^*\alpha'_*F} && {\alpha'^!\alpha'_!F} && {\alpha'_!\alpha'^!\tF}
	\arrow["\simd", tail reversed, from=1-2, to=3-2]
	\arrow["{\mathrm{counit}}"', from=1-4, to=2-3]
	\arrow["\simd", tail reversed, from=1-6, to=3-6]
	\arrow["{\mathrm{counit}}"', from=1-8, to=2-7]
	\arrow["\simd", tail reversed, from=1-8, to=3-8]
	\arrow["{\mathrm{unit}}", from=2-1, to=1-2]
	\arrow["{\mathrm{unit}}"', from=2-1, to=3-2]
	\arrow["{\mathrm{unit}}", from=2-5, to=1-6]
	\arrow["{\mathrm{unit}}"', from=2-5, to=3-6]
	\arrow["\simd"', tail reversed, from=3-4, to=1-4]
	\arrow["{\mathrm{counit}}", from=3-4, to=2-3]
	\arrow["{\mathrm{counit}}", from=3-8, to=2-7]
\end{tikzcd}\]
where the vertical isomorphisms are from Functoriality.
\end{lemma}

\begin{lemma}[Convolution]
    For $A,B\in D(E)$, we have natural isomorphisms:
    $$F(A\ast B)\simeq FA\OX FB\text{ and }F(A\OX B)\simeq (FA\ast FB)\Td.$$
\end{lemma}

\begin{lemma}[Plancherel]
    For $A,B\in D(E)$, we have a natural isomorphism:
    $$\pi_{E'!}(FA\OX FB)\simeq\pi_{E!}(A\OX[-1]^*B)\Tmd.$$
\end{lemma}

In the derived context, the category of monodromic motives $D_{\mathrm{mon}}$ and the homogeneous Fourier transform $F_1$ are defined in the same way as in the end of §\ref{sec_fourier_on_vect}. The (derived) homogeneous Fourier transform has been studied in detail in \cite{khan_homogeneous} and \cite[\nopp §8]{FK}.

\begin{lemma}[Monodromicity]\label{lem_monodromicity_d}
    The Fourier transform restricts to an equivalence $D_{\mathrm{mon}}(E)\isoto D_{\mathrm{mon}}(E')$.
\end{lemma}

\begin{lemma}[Compatibility with the homogeneous Fourier transform]
    We have a commutative diagram
\[\begin{tikzcd}
	{D(E)} & {D(E')} \\
	{D(E/\Gm)} & {D(E'/\Gm)}
	\arrow["F", from=1-1, to=1-2]
	\arrow["{\rho_1^*}", from=2-1, to=1-1]
	\arrow["{F_1}"', from=2-1, to=2-2]
	\arrow["{\rho'^*_1}"', from=2-2, to=1-2]
\end{tikzcd}\]
where $\rho_1$ (resp. $\rho_1'$) is the quotient map $E\rightarrow E/\Gm$ (resp. $E'\rightarrow E'/\Gm$).
\end{lemma}

We have two more compatibilities, proved \textit{mutatis mutandis} as in \cite[\nopp §A.4]{FYZ} and \cite[\nopp §6.4]{FYZ}.

\begin{lemma}[Compatibility with proper base change]
    Given a Cartesian square of derived vector bundles over $X$ and its dual:
\[\begin{tikzcd}[row sep=0.8cm, column sep=0.8cm]
	& D &&& {D'} & \\
	A && C & {A'} && {C'} \\
	& B &&& {B'}
	\arrow["c"', from=1-2, to=2-1]
	\arrow["d", from=1-2, to=2-3]
	\arrow["\square"{marking, allow upside down}, draw=none, from=2-1, to=2-3]
	\arrow["a"', from=2-1, to=3-2]
	\arrow["b", from=2-3, to=3-2]
	\arrow["{c'}", from=2-4, to=1-5]
	\arrow["{d'}"', from=2-6, to=1-5]
	\arrow["\square"{marking, allow upside down}, draw=none, from=2-6, to=2-4]
	\arrow["{a'}", from=3-5, to=2-4]
	\arrow["{b'}"', from=3-5, to=2-6]
\end{tikzcd}\]
Assume $a$ and $b$ admit presentations, then we have a commutative diagram
\[\begin{tikzcd}
	{F_Cb^*a_!} & {F_Cd_!c^*} \\
	{b'_!a'^*F_A\langle d_B-d_C\rangle} & {d'^*c'_!F_A\langle d_A-d_D\rangle}
	\arrow["\simeq"{description}, draw=none, from=1-1, to=1-2]
	\arrow["\simeqd"{description}, draw=none, from=1-1, to=2-1]
	\arrow["\simeqd"{description}, draw=none, from=1-2, to=2-2]
	\arrow["\simeq"{description}, draw=none, from=2-1, to=2-2]
\end{tikzcd}\]
where the vertical isomorphisms are from Functoriality, and the horizontal ones are from proper base change.
\end{lemma}

\begin{lemma}[Compatibility with the Gysin map]
    Let $\alpha: \tE\rightarrow E$ be a quasi-smooth map of derived vector bundles over $X$. Assume $\alpha$ admits a presentation. Then we have a commutative diagram
\[\begin{tikzcd}
	{\alpha'_!F\Td} & {\alpha'_*F\Td} \\
	{\tF\alpha^*\Tdt} & {\tF\alpha^!\Td}
	\arrow["{\mathrm{can}(\alpha')}", from=1-1, to=1-2]
	\arrow["\simeqd"{description}, draw=none, from=1-1, to=2-1]
	\arrow["\simeqd"{description}, draw=none, from=1-2, to=2-2]
	\arrow["{[\alpha]}"', from=2-1, to=2-2]
\end{tikzcd}\]
where $\mathrm{can}(\alpha'): \alpha'_!\rightarrow\alpha'_*$ is the forget-support map,  $[\alpha]: \alpha^*\rightarrow\alpha^!\langle -\dim\alpha\rangle$ is the Gysin map, and the vertical isomorphisms are from Functoriality.
\end{lemma}

Finally, we show the Fourier transform preserves constructible objects $D_c(E)\subseteq D(E)$. See Conventions for basic facts on this notion. 

\begin{lemma}[Constructibility]
    The Fourier transform restricts to an equivalence: $D_{\mathrm{c}}(E)\isoto D_{\mathrm{c}}(E')$. 
\end{lemma}
\begin{proof}
    By Involutivity, it suffices to show that constructibility is preserved by $F$. We first prove the case when $E$ has tor-amplitude in $[0,\infty)$. Let $A\in D_c(E)$, then $FA=p'_!(p^*A\otimes m^*\CL)$ is clearly constructible, noting that $p'$ is representable.\\\\
    We now prove the general case. The question being smooth local on $X$ (this uses that $F$ commutes with base change), we may assume $E$ has a presentation. As before, denote by $E^{\ge0}$ the total space associated to $(\CE^0\rightarrow\CE^1\rightarrow\cdots)$. Let $q: E^{\ge0}\rightarrow E$ be the canonical map, it is a smooth surjection and $q^*: D(E)\rightarrow D(E^{\ge0})$ is fully faithful (Lemma \ref{lem_bootstrap}), the dual map $q'$ is a closed immersion. Let $A\in D_c(E)$. As $q'$ is a closed immersion, $FA$ is constructible if and only if $q'_*FA$ is (because closed immersions are representable by definition, and $q'^*q'_*\isoto\id$). By Functoriality, $q'_*FA\simeq F_{\ge 0}q^!A$, where $F_{\ge 0}$ is the Fourier transform on $E^{\ge0}$. This completes the proof since $F_{\ge 0}q^!A$ is constructible by the first case above. 
\end{proof}

\printbibliography[heading=bibintoc, title={References}]
\Addresses
\end{document}